\documentclass[a4paper,11pt]{amsart}%

\usepackage{amsfonts}
\usepackage{amssymb}
\usepackage[utf8]{inputenc}
\usepackage{amsmath}
\usepackage{xcolor}
\usepackage{graphicx}%
\providecommand{\U}[1]{\protect\rule{.1in}{.1in}}
\usepackage[colorlinks=true, linkcolor=red, citecolor=blue]{hyperref}
\usepackage{mathtools}
\mathtoolsset{showonlyrefs}
\numberwithin{equation}{section}

\newcommand{\R}{\mathbb{R}}

\newtheorem{theorem}{Theorem}[section]
\newtheorem{proposition}[theorem]{Proposition}
\newtheorem{lemma}[theorem]{Lemma}
\newtheorem{corollary}[theorem]{Corollary}
\newtheorem{remark}{Remark}

\theoremstyle{remark}

\newcommand{\remove}[1]{ }

\begin{document}
\title[Control for 4NLS]{Stabilization and control  for the biharmonic Schr\"odinger equation}
\author[Capistrano--Filho]{Roberto A. Capistrano--Filho}
\address{Departmento de Matem\'atica,  Universidade Federal de Pernambuco (UFPE), 50740-545, Recife (PE), Brazil.}
\email{capistranofilho@dmat.ufpe.br}
\author[Cavalcante]{Márcio Cavalcante}
\address{Instituto de Matem\'atica, Universidade Federal de Alagoas (UFAL), Macei\'o (AL), Brazil.}
\email{marcio.melo@im.ufal.br}
\date{July, 2018}

\begin{abstract}
The main purpose of this paper is to show the global stabilization and exact controllability  properties for a fourth order nonlinear fourth order nonlinear Schr\"odinger system on a periodic domain $\mathbb{T}$ with internal control supported on an arbitrary sub-domain of $\mathbb{T}$. More precisely, by certain properties of propagation of compactness and regularity in Bourgain spaces, for the solutions of the associated linear system, we show that the system is globally exponentially stabilizable. This property together with the local exact controllability ensures that fourth order nonlinear Schr\"odinger is globally exactly controllable.
\end{abstract}
\maketitle

\section{Introduction\label{sec0}}

\subsection{Presentation of the model} Fourth-order nonlinear Schr\"odinger (4NLS) equation or biharmonic cubic nonlinear Schr\"odinger equation 
\begin{equation}
\label{fourtha}
i\partial_tu +\partial_x^2u-\partial_x^4u=\lambda |u|^2u,
\end{equation}
have been introduced by Karpman \cite{Karpman} and Karpman and Shagalov \cite{KarSha} to take into account the role of small fourth-order dispersion terms in the propagation of intense laser beams in a bulk medium with Kerr nonlinearity. Equation \eqref{fourtha} arises in many scientific fields such as quantum mechanics, nonlinear optics and plasma physics, and has been intensively studied with fruitful references (see \cite{Ben,CuiGuo,Karpman,Paus,Paus1} and references therein).

The past twenty years such 4NLS have been deeply studied from differents mathematical viewpoint. For example, Fibich \textit{et al.} \cite{FiIlPa} worked various properties of the equation in the subcritical regime, with part of their analysis relying on very interesting numerical developments.  The well-posedness and existence of the solutions has been shown (see, for instance, \cite{Paus,Paus1,tsutsumi,Tzvetkov,WenChaiGuo1}) by means of the energy method, harmonic analysis, etc.  

It is interesting to point out that there are many works related with the equations \eqref{fourth} not only dealing with well-posedness theory. For example, recently Natali  and Pastor \cite{NataliPastor}, considered the fourth-order dispersive cubic nonlinear Schr\"odinger equation on the line with mixed dispersion. They proved the orbital stability, in the $H^2(\mathbb{R})$--energy space by constructing a suitable Lyapunov function. Considering the  equation \eqref{fourth} on the circle, Oh and Tzvetkov  \cite{Tzvetkov}, showed that the mean-zero Gaussian measures on Sobolev spaces $H^s (\mathbb{T})$, for $s>\frac{3}{4}$, are quasi-invariant under the flow. For instance, in this spirit, there has been a significant progress over the recent years and the reader can have a great view in to nonlinear Schr\"odinger equation in \cite{Burq,Burq1}.

\subsection{Setting of the problem}

In this article our purpose is to study properties of stabilization and, consequently, controllability for the periodic one-dimensional fourth order nonlinear dispersive Schr\"odinger equation:
\begin{equation}\label{fourth}
\begin{cases}
i\partial_tu +\partial_x^2u-\partial_x^4u
=\lambda |u|^2u,& (x,t)\in \mathbb{T}\times \mathbb{R},\\
u(x,0)=u_0(x),& x\in \mathbb{T}.\end{cases}
\end{equation}
To give properties of controllability of the system \eqref{fourth} in large time for a control supported in any small open subset of $\mathbb{T}$ we will study the equation \eqref{fourth} from a control point of view with a forcing term $f=f(x,t)$ added to the equation as a control input 
\begin{equation}\label{fourthC}
\begin{cases}
i\partial_tu +\partial_x^2u-\partial_x^4u
=\lambda |u|^2u +f,& (x,t)\in \mathbb{T}\times \mathbb{R},\\
u(x,0)=u_0(x),& x\in \mathbb{T},\end{cases}
\end{equation}
where $f$ is assumed to be supported in a given open subset $\omega$ of $\mathbb{T}$. Therefore,  the following classical issues  related with the control theory are considered in this work:

\vglue 0.2 cm
\noindent\textbf{Exact control problem:}  Given an initial state $u_0$ and a terminal state $u_1$ in a certain space, can one find an appropriate control input $f$ so that the equation \eqref{fourthC} admits a solution $u$ which satisfies $u(\cdot,0)= u_0$ and $u(\cdot,T)= u_1$ ?

\vglue 0.2 cm
\noindent\textbf{Stabilization problem:}  Can one find a feedback control law $f$ so that the system \eqref{fourthC} is asymptotically stable as $t\to\infty $?

\vglue 0.2 cm



\subsection{Previous results} When we consider equation \eqref{fourtha} on a periodic domain $\mathbb{T}$ is not of our knowledge any result about control theory. However, there are interesting results on a bounded domain of $\mathbb{R}$ or $\mathbb{R}^n$, which we will summarize on the paragraphs below.

The first result about the exact controllability of 4NLS \eqref{fourtha} on a bounded domain $\Omega$ of the $\mathbb{R}^n$ is due to Zheng and  Zhongcheng in \cite{zz}. In this work, by means of an $L^2$--Neumann boundary control, the authors proved that the solution is exactly controllable in $H^s(\Omega)$, $s=-2$, for an arbitrarily small time. They used Hilbert Uniqueness Method (HUM) (see, for instance, \cite{DolRus1977,lions1}) combined with the multiplier techniques to get the main result of the article. More recently,  in \cite{zheng}, Zheng proved another interesting problem related with a control theory. He showed a global Carleman estimate for the fourth order Schr\"odinger equation posed on a $1- d$ finite domain. The Carleman estimate is used to prove the Lipschitz stability for an inverse problem consisting in retrieving a stationary potential in the Schr\"odinger equation from boundary measurements.

Still on control theory Wen \textit{et. al}, in two works \cite{WenChaiGuo1,WenChaiGuo}, studied well-posedness and control theory related with the equation \eqref{fourtha} on a bounded domain of $\mathbb{R}^n$, for $n\geq2$. In \cite{WenChaiGuo1}, they proved the Neumann boundary controllability with collocated observation. With this result in hads, the exponential stability of the closed-loop system under proportional output feedback control holds. Recently, the authors, in \cite{WenChaiGuo}, gave a positive answers when considered the equation with hinged boundary by either moment or Dirichlet boundary control and collocated observation, respectively.

Lastly, to get a general outline of the control theory already done for the system \eqref{fourtha}, two interesting problems were studied recently by Aksas and Rebiai \cite{AkReSa} and Peng \cite{Peng}: Stochastic control problem and uniform stabilization, in a smooth bounded domain $\Omega$ of $\mathbb{R}^n$ and on the interval $I=(0,1)$ of $\mathbb{R}$, respectively. The first work, by introducing suitable dissipative boundary conditions, the authors proved that the solution decays exponentially in $L^2(\Omega)$ when the damping term is effective on a neighborhood of a part of the boundary. The results are established by using multiplier techniques and compactness/uniqueness arguments. With regard of the second work, above mentioned, the author showed a Carleman estimates for forward and backward stochastic fourth order Schr\"odinger equations which provided to prove the observability inequality,  unique continuation property and, consequently,  the exact controllability for the forward and backward stochastic system associated to \eqref{fourtha}.

\subsection{Notations and main results} Before to present our main results, let us introduce the Bourgain spaces associated to the 4NLS \eqref{fourtha}. For given $b,s\in\mathbb{R}$ and a function $u:\mathbb{T\times
R\to R}$, defines the quantity%
\[
\left\Vert u\right\Vert _{X_{b,s}}:=\left(  \sum\limits_{k=-\infty}^{\infty
}\int_{\mathbb{R}}\left\langle k\right\rangle ^{2s}\left\langle \tau-p\left(
k\right)  \right\rangle ^{2b}\left\vert \hat{u}\left(  k,\tau\right)
\right\vert ^{2}d\tau\right)  ^{\frac{1}{2}}\text{,}%
\]%
where $\hat{u}\left(  k,\tau\right)  $ denotes de Fourier transform of $u$
with respect to the space variable $x$ and the time variable $t$,
$\left\langle \cdot\right\rangle =\sqrt{1+\left\vert \text{ }\cdot\text{
}\right\vert ^{2}}$ and $p\left(  k\right)  =k^{4}-k^{2}$. We 
denote by $D^{r}$ the operator defined on $\mathcal{D}^{\prime}\left(
\mathbb{T}\right)  $ by%
\begin{equation}\label{k10}
\widehat{D^{r}u}\left(  k\right)  =\left\{
\begin{array}
[c]{ll}%
\left\vert k\right\vert ^{r}\hat{u}\left(  k\right)  & \text{if }%
k\neq0\text{,}\\
\hat{u}\left(  0\right)  & \text{if }k=0\text{.}%
\end{array}
\right.  
\end{equation}
The Bourgain space $X_{b,s}$  associated to the fourth order linear dispersive Schr\"odinger equation on $\mathbb{T}$ is the completion of the Schwartz space
$\mathcal{S}\left(  \mathbb{T\times R}\right)  $ under the norm $\left\Vert
u\right\Vert _{X_{b,s}}$. Note
that for any $u\in X_{b,s}$,%
\[
\left\Vert u\right\Vert _{X_{b,s}}=\left\Vert W\left(  -t\right)  u\right\Vert
_{H^{b}\left(  \mathbb{R},H^{s}\left(  \mathbb{T}\right)  \right)
}\text{.}%
\]
For a given interval $I$, let $X_{b,s}\left(  I\right)  $  be the restriction space of $X_{b,s}$ to the
interval $I$ with the norm
\[
\left\Vert u\right\Vert _{X_{b,s}}=\inf\left\{  \left\Vert \tilde
{u}\right\Vert _{X_{b,s}}\text{ }|\text{ }\tilde{u}=u\text{ on }%
\mathbb{T}\times I\right\}.
\]%
By simplicity, we denote $X_{b,s}\left(  I\right)  $ by $X_{b,s}^{T}$  when $I=\left(  0,T\right)$. 

To clarify, the first issue to be proved in this article is the following one.

\vspace{0.1cm}

\noindent\textit{Given $T>0$ and $u_0,u_1\in L^2(\mathbb{T})$, does there exist a control input $g\in C([0,T];L^2(\mathbb{T}))$ in order to make the solution of 
\begin{equation}\label{Controlzeroa}
	\begin{cases}
	i\partial_tu +\partial_x^2u-\partial_x^4u	=\lambda |u|^2u+g,& (x,t)\in \mathbb{T}\times \mathbb{R},\\
	u(x,0)=u_0(x),& x\in \mathbb{T}\end{cases}
	\end{equation}
satisfy $u(\cdot,T) = u_1$?}

\vspace{0.1cm}

The strategy to answer this question is first to prove a local exact controllability result and to combine it with a global stabilization of the solutions to ensure the global controllability of the system \eqref{Controlzeroa}. Thus, in this spirit, first main result of the article is concerned with the control property of 4NLS \eqref{Controlzeroa} near to $0$, that will be proved using a perturbation argument introduced by Zuazua in \cite{Zuazua}. More precisely, we will show the following local controllability:

\begin{theorem}\label{main1}
Let  $\omega$ be any nonempty open subset of $\mathbb{T}$ and $T>0$. Then, there exist $\epsilon>0$ and $\eta>0$ such that for any $u_0\in L^2(\mathbb{T})$ with $$\left\Vert u_0 \right\Vert _{L^2(\mathbb{T})}<\epsilon,$$ there exists $g\in C([0,T];L^2(\mathbb{T}))$, with $\left\Vert g \right\Vert _{L^{\infty}([0,T];L^2(\mathbb{T}))}<\eta$ and $supp(g)\subset\omega\times(0,T)$ such that the unique solution $u\in X^T_{b,0}$ of the system \eqref{Controlzeroa}
	satisfies $u(x,T)=0$. In addition, if $u_0\in H^s(\mathbb{T})$, with $s\geq0$, with a large $H^s-$norm, then $g\in C([0,T];H^s(\mathbb{T}))$.
\end{theorem}

To introduce our second main result, let $a(x)\in L^{\infty}(\mathbb{T})$ real valued, the stabilization system that we will consider is the following
\begin{equation}
\left\{
\begin{array}
[c]{lll}%
i\partial_{t}u+\partial_{x}^{2}u-\partial_{x}^{4}u+ia^2u=\lambda \left\vert
u\right\vert^2u &  & \text{on }\mathbb{T\times}\left(
0,T\right)  \text{,}\\
u\left(  x,0\right)  =u_{0}\left(  x\right)  &  & \text{on }\mathbb{T}\text{,}%
\end{array}
\right.  \label{k35a}%
\end{equation}
where $\lambda\in\mathbb{R}$ and $u_{0}\in L^{2}\left(\mathbb{T}\right)  $, in $L^2$--level.

Note that, easily, we can check that the solution of \eqref{k35a} satisfies the mass decay
\begin{equation}\label{energyint}
\left\Vert u\left(\cdot,  t\right)  \right\Vert _{L^2(\mathbb{T})}^{2}=\left\Vert u\left(\cdot,0\right)  \right\Vert _{L^2(\mathbb{T})}^{2}-\int_{0}^t\left\Vert au(\tau)\right\Vert _{L^2(\mathbb{T})}^{2}d\tau, \quad\forall t\geq0.
\end{equation}
Observe that for $a(x)=0$ we have, by \eqref{energyint}, the mass of the system is indeed conserved. However, assuming that $a(x)^2>\eta> 0$ on some nonempty open set $\omega$ of $\mathbb{T}$, identity \eqref{energyint} states that we have an possibility of a exponential decay of the solutions related of \eqref{k35a}. In fact,  following the ideas of Dehman and Lebeau \cite{dehman-lebeau}, see also \cite{dehman-gerard-lebeau}, by using techniques of semiclassical and microlocal analysis, the result that we are able to prove, for large data, can be read as follows:

\begin{theorem}\label{main} Let $a(x):=a\in L^{\infty}(\mathbb{T})$ taking real values such that $a^2(x)>\eta$ on a nonempty open set $\omega\subset\mathbb{T}$, for some constant $\eta>0$. Then, for any $R_{0}>0$ and $u_{0}\in L^{2}\left(  \mathbb{T}\right)  $ with
\[
\left\Vert u_{0}\right\Vert _{L^2(\mathbb{T})}\leq R_{0}\text{,}%
\]
there exist $\gamma>0$ and $C>0$ such that the corresponding solution $u$ of
(\ref{k35a}) satisfies
\begin{equation}
\left\Vert u\left(  \cdot,t\right)  \right\Vert _{L^2(\mathbb{T})}\leq Ce^{-\gamma
t}\left\Vert u_{0}\right\Vert _{L^2(\mathbb{T})}\text{, }\forall t>0\text{.} \label{k45}
\end{equation}
\end{theorem}

Finally, the global controllability result which one can be established is the following:

\begin{theorem}\label{main2}
Let $1/2\leq b<21/16$, $\omega$ be any nonempty open subset of $\mathbb{T}$ and $R_0>0$. Then, there exist $T>0$ and $C>0$ such that for any $u_0$ and $u_1$ in $L^2(\mathbb{T})$ with $$\left\Vert u_0 \right\Vert _{L^2(\mathbb{T})}<R_0 \quad \text{and} \quad \left\Vert u_1 \right\Vert _{L^2(\mathbb{T})}<R_0$$ there exists $g\in C([0,T];L^2(\mathbb{T}))$, with $\left\Vert g \right\Vert _{L^{\infty}([0,T];L^2(\mathbb{T}))}<\eta$, with $supp(g)\subset\omega\times(0,T)$, such that the unique solution $u\in X^T_{b,0}$ of the system \eqref{Controlzeroa}
	satisfies $u(x,T)=u_1$. In addition, if $u_0$ and $u_1$ belongs to $H^s(\mathbb{T})$, $s\geq0$, with a large $H^s-$norm, then $g\in C([0,T];H^s(\mathbb{T}))$.
\end{theorem}

Let us describe briefly the main arguments of the proof of these theorems.  Precisely, the control result for large data (Theorem \ref{main2}) will be a combination of a global stabilization result (Theorem \ref{main}) and the local control result (Theorem \ref{main1}), as is usual in control theory, see e.g., \cite{dehman-lebeau-zuazua,dehman-gerard-lebeau,Laurent-esaim,LaurentLinaresRosier,Laurent}. 

With respect to the proof of Theorem \ref{main}, in general lines, first, the functional spaces used here are the Bourgain spaces which are especially suited for solving dispersive equations. Thus, the step one is to prove the following \textit{strichartz estimate} for the operator of fourth order Scr\"odinger equation:
$$
\left\Vert u\right\Vert _{L^{4}\left(  [0,T]\times\mathbb{T}
\right)  }\leq C\left\Vert u\right\Vert _{X_{\frac{5}{16},0}^{T}}\text{,}
$$ this allows to prove the following \textit{multilinear estimates} in $X_{b,s}^T$:
\begin{equation*}
\begin{split}
\left\Vert   u_1u_2\overline{u}_3\right  \Vert _{X_{b-1,s}^{T}}&\leq
C\left\Vert u_1\right\Vert _{X_{\frac{5}{16},0}^{T}}\left\Vert u_2\right\Vert_{X_{\frac{5}{16},0}^{T}}\left\Vert u_3\right\Vert_{X_{\frac{5}{16},s}^{T}}\text{,}\\
\left\Vert  |u|^2u- |v|^2v\right  \Vert _{X_{b-1,s}^{T}}&\leq
C(\left\Vert u \right\Vert _{X_{\frac{5}{16},s}^{T}}^2+\left\Vert v\right\Vert_{X_{\frac{5}{16},s}^{T}})^2\left\Vert u-v\right\Vert_{X_{\frac{5}{16},s}^{T}}\text{,}
\end{split}
\end{equation*}
where $s\geq 0$, $b<\frac{21}{16}$ and $T\leq 1$ are given and $C(s):=C>0$.  On the step two, $H^s(\mathbb{T})$ propagation of regularity and compacity (see section \ref{sec5} below)  from the state to the control is obtained using this property for the linear control and a local linear behavior. Lastly, results of propagation together with the \textit{unique continuation property (UCP)}, bellow presented, guarantees the proof of Theorem \ref{main}. 
\begin{proposition}\label{UCP_int}
For every $T>0$ and $\omega$ any nonempty open set of $\mathbb{T}$, the only solution $u\in C^{\infty}([0,T]\times\mathbb{T})$ of the system%
\[
\left\{
\begin{array}
[c]{lll}%
i\partial_{t}u+\partial_{x}^{2}u-\partial_{x}^{4}u=b(x,t)u&  & \text{on }\mathbb{T\times}\left(  0,T\right)
\text{,}\\
u=0 &  & \text{on }\omega\times\left(  0,T\right)  \text{,}%
\end{array}
\right.
\]
where $b(x,t)\in C^{\infty}([0,T]\times\mathbb{T})$, is the trivial one 
\[
u\left(  x,t\right)  =0\text{ \ on \ }\mathbb{T\times}\left(  0,T\right)
\text{.}%
\]
\end{proposition}

To end our introduction, we present the outline of our paper as follows: Section \ref{sec2} is for establish estimates needed in our analysis, namely, \textit{Strichartz estimates} and \textit{trilinear estimates}. Existence of solution for 4NLS with source and damping term will be presented in Section \ref{sec3}. In Section \ref{sec4} we prove the local controllability result, Theorem \ref{main1}. Next, Section \ref{sec5}, the \textit{propagation of compactness and regularity in Bourgain space} are proved and, with this in hands, the Section \ref{sec6} is aimed to present the proof of unique continuation property, Proposition \ref{UCP_int}. Finally, Section \ref{sec7}, is devoted to prove Theorem \ref{main}.

\section{Linear Estimates\label{sec2}}
In this section we introduce some results which are essential to establish the exact
controllability and stabilization of the nonlinear system \eqref{Controlzeroa} and \eqref{k35a}, respectively. 

%
%
%
%
%

\subsection{Strichartz  and trilinear estimates}

%

The next estimate is a Strichartz type estimate.

\begin{lemma}\label{strichartz_estimates}The following estimate holds
\begin{equation}
\left\Vert u\right\Vert _{L^{4}\left(  \mathbb{T}\times\mathbb{R}
\right)  }\leq C\left\Vert u\right\Vert _{X_{\frac{5}{16},0}^{T}}\text{.}
\label{k12}%
\end{equation}
\end{lemma}

\begin{proof}
		We closely follow the argument for the $L^4$-Strichartz estimate for the usual
	(second and forth order) Schr\"odinger equation presented in \cite{Tao} and \cite{Tzvetkov}.
	
Given dyadic $M\geq 1$, let $u_M$ the restriction of $u$ onto the modulation size $\langle \tau-p(k)\rangle \sim M$. Then, it suffices to show that there exists $\epsilon>0$ such that
\begin{equation}\label{assume}
\| u_Mu_{2^mM}\|_{L_x^2L_{t}^2}\leq 2^{-\epsilon m}M^{\frac{5}{16}}\|u_M\|_{L_{x,t}^2}(2^mM)^{5/16}\|u_{2^mM}\|_{L_{x,t}^2}
\end{equation}
for any $M\geq 1$ and $m \in \mathbb{N}\cup \{0\}$.
Indeed, assuming \eqref{assume}, by Cauchy-Schwarz inequality, we have that 
\begin{equation*}
\begin{split}
\left\Vert u\right\Vert _{L^{4}\left(  \mathbb{T}\times\mathbb{R}
	\right)  }^2&= \sum_M\sum_{m\geq 0} \|u_Mu_{2^mM}\|_{L_{x,t}^2}\\
& \lesssim \sum_M\sum_{m\geq 0} 2^{-\epsilon m}M^{\frac{5}{16}}\|u_M\|_{L_{x,t}^2}(2^mM)^{5/16}\|u_{2^mM}\|_{L_{x,t}^2}\\
&\lesssim \sum_{m\geq 0} 2^{-\epsilon m}\left(\sum_M M^{\frac{5}{8}}\|u_M\|_{L_{x,t}^2}^2\right)^{1/2}\left(\sum_M(2^mM)^{5/8}\|u_{2^mM}\|_{L_{x,t}^2}^2\right)^{1/2}\\
&\lesssim \left\Vert u\right\Vert _{X_{\frac{5}{16},0}^{T}}^2.
\end{split}
\end{equation*}
This proves \eqref{k12}. 

Now we prove \eqref{assume}. By Plancherel’s identity and H\"older’s inequality, the following inequality follows
\begin{equation*}
\begin{split}
\| u_Mu_{2^mM}\|_{L_x^2L_{t}^2}=&\left\|\sum_{k=k_1+k_2}\int_{\tau=\tau_1+\tau_2}\hat{u}_M(k_1,\tau_1)\hat{u}_{M2^m}(k_2,\tau_2)\right\|_{l^2_kL^2_{\tau}}\\
&\lesssim \left( \sup_{k,\tau}A(k,\tau)\right)^{1/2} \|u_M\|_{L_{x,t}^2}\|u_{2^mM}\|_{L_{x,t}^2},
\end{split}
\end{equation*}
where $A(k,\tau)$ is given by 
\begin{equation}\label{ak}
A(k,\tau)=\sum_{n=n_1+n_2}\int_{\tau=\tau_1+\tau_2} \textbf{1}_{\tau_1-p(k_1)\sim M}\textbf{1}_{\tau_2-p(k_2)\sim 2^mM}d\tau_1.
\end{equation}
Integrating in $\tau_1$ holds that
\begin{equation*}
A(k,\tau)\leq M \sum_{n=n_1+n_2}\textbf{1}_{\tau\sim-p(k_1)-p(k_2) +2^mM}.
\end{equation*} 
Here, we have used that the Lebesgue  measure of set $\{\tau_1\in \mathbb{R}; \tau_1-p(k_1)\sim M \}$ is comparable with $M$ and $\tau=\tau_1+\tau_2\sim p(k_1)+p(k_2)+M+2^mM$. A simple analysis proves that
\begin{equation*}
\sum_{n=n_1+n_2}\textbf{1}_{\tau\sim-p(k_1)-p(k_2) +2^mM}\leq (2^mM)^{1/4}.
\end{equation*}
Thus, from \eqref{ak}, we have
\begin{equation*}
A(k,\tau)\lesssim 2^{m/8}M^{1/8}\leq 2^{-\frac{3m}{16}} (2^{m}M)^{5/16}M^{5/16}.
\end{equation*}
This finish the proof of the Lemma.
\end{proof}
Lemma \ref{strichartz_estimates} allow us to prove the following multilinear estimates in $X^T_{b,s}$.
\begin{lemma}\label{trilinear_estimates}Let $s\geq 0$, $b\geq\frac{5}{16}$ and $T\leq 1$ be
given. There exist a constant $C(s):=C>0$ such that the following trilinear estimates
holds%
\begin{equation*}
\begin{split}
\left\Vert   u_1u_2\overline{u}_3\right  \Vert _{X_{-b,s}^{T}}&\leq
C\left\Vert u_1\right\Vert _{X_{\frac{5}{16},0}^{T}}\left\Vert u_2\right\Vert_{X_{\frac{5}{16},0}^{T}}\left\Vert u_3\right\Vert_{X_{\frac{5}{16},s}^{T}}\text{,}\\ 
\left\Vert  |u|^2u- |v|^2v\right  \Vert _{X_{-b,s}^{T}}&\leq
C(\left\Vert u \right\Vert _{X_{\frac{5}{16},s}^{T}}^2+\left\Vert v\right\Vert_{X_{\frac{5}{16},s}^{T}})^2\left\Vert u-v\right\Vert_{X_{\frac{5}{16},s}^{T}}\text{.}
\end{split}
\end{equation*}
Moreover, there exist constants $C,\ C(s):=C_s>0$ independent on $T\leq1$ such that for every $s\geq 1$, 
follows that
\begin{equation*}
\left\Vert   |u|^2u\right  \Vert _{X_{-b,s}^{T}}\leq
C\left\Vert u\right\Vert _{X_{\frac{5}{16},0}^{T}}^2\left\Vert u\right\Vert_{X_{\frac{5}{16},s}^{T}}+C_s\left\Vert u\right\Vert _{X_{\frac{5}{16},s-1}^{T}}\left\Vert u\right\Vert _{X_{\frac{5}{16},1}^{T}}\left\Vert u\right\Vert_{X_{\frac{5}{16},0}^{T}}\text{.} 
\end{equation*}
\end{lemma}

\begin{proof}
Here we will use the ideais contained in \cite{Bourgain}. Let $w=u_1u_2\overline{u}_3$, by duality we have that
\begin{equation}\label{tec1}
\begin{split}
\|w\|_{X_{-b,s}^{T}}&=\sup_{\|c\|_{l_{k}^2L_{\tau}^2}\leq 1}\sum_{k=-\infty}^{+\infty}\int_{\tau}\langle k\rangle^{s}\langle\tau-p(k)\rangle^{-b}\hat{w}(k,\tau)c(k,\tau)d\tau\\
&=\sup_{\|c\|_{l_{k}^2L_{\tau}^2}\leq 1}\sum_{k,k_2,k_3=-\infty}^{+\infty}\int_{\tau,\tau_2,\tau_3}\langle k\rangle^{s}\langle\tau-p(k)\rangle^{-b}\hat{u}_1(k_1,\tau_1)\hat{u}_2(k_2,\tau_2)\hat{\overline{u}}_3(k_3,\tau_3)c(k,\tau)d\tau\tau_2\tau_3
\end{split}
\end{equation}
where $k=k_1+k_2-k_3$ and $\tau=\tau_1+\tau_2-\tau_3$.  

Assume that $\max\{k_1,k_2,k_3\}=k_1$ and define $\hat{f}(k,\tau)= c(k,\tau)^{s}\langle\tau-p(k)\rangle^{-b} $ and $\hat{v}(k_1,\tau_1)=\langle k_1\rangle^{s}\hat{u}_1(k_1,\tau_1)$. Then, last expression of \eqref{tec1} is bounded by
\begin{equation*}
\begin{split}
3\sup_{\|c\|_{l_{k}^2L_{\tau}^2}\leq 1}\sum_{k,k_2,k_3=-\infty}^{+\infty}&\int_{\tau,\tau_2,\tau_3}\hat{f}(\xi,\tau)\hat{v}_1(\xi_1,\tau_1)\hat{u}_2(k_2,\tau_2)\hat{\overline{u}}_3(k_3,\tau_3)d\tau\tau_2\tau_3,\\
&\leq 3 \|f\ v\ u_2\ u_3\|_{L_{\tau}^1l_k^1}\\
&\leq 3 \sup_{\|c\|_{l_{k}^2L_{\tau}^2}\leq 1} \|f\|_{L_{\tau}^4l_k^4}\|v\|_{L_{\tau}^4l_k^4}\|u_2\|_{L_{\tau}^4l_k^4}\|u_3\|_{L_{\tau}^4l_k^4}\\
&\leq 3\sup_{\|c\|_{l_{k}^2L_{\tau}^2}\leq 1} \|f\|_{X_{\frac{5}{16}},0}\|v\|_{X_{\frac{5}{16}},0}\|u_2\|_{X_{\frac{5}{16}},0}\|u_3\|_{X_{\frac{5}{16}},0}\\
&\leq 3\sup_{\|c\|_{l_{k}^2L_{\tau}^2}\leq 1} \|c\|_{L_{\tau}^2l_k^2}\|u_1\|_{X_{\frac{5}{16},0}}\|u_2\|_{X_{\frac{5}{16},0}}\|u_3\|_{X_{\frac{5}{16},0}},
\end{split}
\end{equation*}
where we have used Lemma \ref{strichartz_estimates} and the fact $b>\frac{5}{16}$.
\end{proof}

\subsection{Auxiliary lemmas} This subsection is devoted to present auxiliaries results related to the Bourgain space $X_{b,s}$ which are used several times in this work and played important roles on the main results of this article.

\begin{lemma}
\label{estimates_a}Let $-1\leq b\leq1$, $s\in\mathbb{R}$ and $\varphi\in
C^{\infty}\left(  \mathbb{T}\right)  $. Then, for any $u\in X_{b,s}$,
$\varphi\left(  x\right)  u\in X_{b,s-3\left\vert b\right\vert }$. Similarly,
the multiplication by $\varphi$ maps $X_{b,s}^{T}$ into $X_{b,s-3\left\vert
b\right\vert }^{T}$.
\end{lemma}

\begin{proof}
We first consider the case of $b=0$ and $b=1$. The other cases of $b$ will be
derived later by interpolation and duality.

For $b=0$,
\[
X_{0,s}=L^{2}\left(  \mathbb{R};H^{s}\left(  \mathbb{T}\right)
\right)
\]
and the result is obvious. For $b=1$, we have $u\in X_{1,s}$ if and only if%
\[%
\begin{array}
[c]{ccc}%
u\in L^{2}\left(  \mathbb{R};H^{s}\left(  \mathbb{T}\right)  \right)  & \text{
and } & i\partial_tu +\partial_x^2u-\partial_x^4u\in L^{2}\left(  \mathbb{R};H^{s}\left(  \mathbb{T}\right)
\right)  \text{,}%
\end{array}
\]
with the norm%
\[
\left\Vert u\right\Vert _{X_{1,s}}^{2}=\left\Vert u\right\Vert _{L^{2}\left(
\mathbb{R};H^{s}\left(  \mathbb{T}\right)  \right)  }^{2}+\left\Vert
i\partial_tu +\partial_x^2u-\partial_x^4u\right\Vert _{L^{2}\left(  \mathbb{R};H^{s}\left(  \mathbb{T}%
\right)  \right)  }^{2}\text{.}%
\]
Thus,%
\begin{align*}
\left\Vert \varphi\left(  x\right)  u\right\Vert _{X_{1,s-3}}^{2}  
=&\ \left\Vert \varphi u\right\Vert _{L^{2}\left(  \mathbb{R};H^{s-3}\left(
\mathbb{T}\right)  \right)  }^{2}+\left\Vert i\partial_t(\varphi u) +\partial_x^2(\varphi u)-\partial_x^4(\varphi u)
\right\Vert_{L^{2}\left(  \mathbb{R};H^{s-3}\left(  \mathbb{T}\right)
\right)  }^{2}\\
  \leq &\ C\left(  \left\Vert u\right\Vert _{L^{2}\left(  \mathbb{R};H^{s-3}\left(  \mathbb{T}\right)  \right)  }^{2}+\left\Vert \varphi\left(i\partial_tu +\partial_x^2u-\partial_x^4u\right)
\right\Vert _{L^{2}\left(  \mathbb{R};H^{s-3}\left(  \mathbb{T}\right)
\right)  }^{2}\right) \\
&  +\left\Vert \left[  \varphi,\partial_x^2-\partial_x^4\right]  u\right\Vert _{L^{2}\left(  \mathbb{R};H^{s-3}\left(  \mathbb{T}\right)  \right)  }^{2}\\
 \leq&\  C\left(  \left\Vert u\right\Vert _{L^{2}\left(  \mathbb{R};H^{s-3}\left(  \mathbb{T}\right)  \right)  }^{2}+\left\Vert i\partial_tu +\partial_x^2u-\partial_x^4u\right\Vert
_{L^{2}\left(  \mathbb{R};H^{s-4}\left(  \mathbb{T}\right)  \right)  }%
^{2}+\left\Vert u\right\Vert _{L^{2}\left(  \mathbb{R};H^{s}\left(
\mathbb{T}\right)  \right)  }^{2}\right) \\
\leq &\  C\left\Vert u\right\Vert _{X_{1,s}}^{2}\text{.}%
\end{align*}
Here, we have used the fact%
\begin{align*}
 \left[  \varphi,\partial_x^2-\partial_x^4\right]
&  =-4\left(  \partial_{x}^{3}\varphi\right)  \partial_{x}-12\left(
\partial_{x}^{2}\varphi\right)  \partial_{x}^{2}+2\partial_x \varphi \partial_x-4\partial_x\varphi \partial_x  +\left(  -\partial_{x}^{4}%
\varphi+\partial_{x}^{2}\varphi\right)
\end{align*}
is a differential operator of order $3$.

To conclude, we prove that the $X_{b,s}$ spaces are in interpolation. Using
Fourier transform, $X_{b,s}$ may be viewed as the weighted $L^{2}$ space
$L^{2}\left(  \mathbb{R}_{\tau}\times\mathbb{Z}_{k},\left\langle
k\right\rangle ^{2s}\left\langle \tau-k^{4}+ k^{2}\right\rangle
^{2b}\lambda\otimes\delta\right)  $, where $\lambda$ is the Lebesgue measure
on $\mathbb{R}$ and $\delta$ is the discrete measure on $\mathbb{Z}$. Then, we
use the complex interpolation theorem of Stein-Weiss for weighted $L^{p}$
spaces (see \cite[p. 114]{Berg}): For $\theta\in\left(  0,1\right)  $%
\[
\left(  X_{0,s},X_{1,s^{\prime}}\right)  _{\left[  \theta\right]  }\approx
L^{2}\left(  \mathbb{R}\times\mathbb{Z},\left\langle k\right\rangle
^{2s\left(  1-\theta\right)  +2s^{\prime}\theta}\left\langle \tau-k^{4}+ k^{2}\right\rangle ^{2\theta}\mu\otimes\delta\right)  \approx
X_{\theta,s\left(  1-\theta\right)  +s^{\prime}\theta}\text{.}%
\]
Since the multiplication by $\varphi$ maps $X_{0,s}$ into $X_{0,s}$ and
$X_{1,s}$ into $X_{1,s-3}$, we conclude that for $b\in\left[  0,1\right]  $,
it maps $X_{b,s}=\left(  X_{0,s},X_{1,s}\right)  _{\left[  b\right]  }$ into
$\left(  X_{0,s},X_{1,s-3}\right)  _{\left[  b\right]  }=X_{b,s-3b}$, which
yields the $3b$ loss of regularity as announced.

Then, by duality, this also implies that for $b\in\left[  0,1\right]  $, the
multiplication by $\varphi\left(  x\right)  $ maps $X_{-b,-s+3b}$ into
$X_{-b,-s}$. As the number $s$ may take arbitrary values in $\mathbb{R}$, we
also have the result for $b\in\left[  -1,0\right]  $ with a loss of
$-3b=3\left\vert b\right\vert $. Finally, to get the same result for the
restriction spaces $X_{b,s}^{T}$, consider%
\[
\tilde{u}=\left\{
\begin{array}
[c]{lll}%
u, &  & \text{if }x\in\mathbb{T}\text{,}\\
0, &  & \text{other cases,}%
\end{array}
\right.
\]
thus%
\[
\left\Vert \varphi u\right\Vert _{X_{b,s-3\left\vert b\right\vert }^{T}}%
\leq\left\Vert \varphi\tilde{u}\right\Vert _{X_{b,s-3\left\vert b\right\vert
}}\leq C\left\Vert \tilde{u}\right\Vert _{X_{b,s}}\text{.}%
\]
Taking the infimun on all the $\tilde{u}$, the result is archived.
\end{proof}

Finally, to close this section, more four auxiliaries lemmas are enunciated. We follow \cite{Ginibre}, where the reader can also find the proofs, thus will be omitted it.

\begin{lemma}
\label{estimates}Let $\frac12\leq b<1$, $s\in\mathbb{R}$ and $T>0$ be given. Then, there exists a constant $C>0$ such that:

(i) For any $\phi\in H^{s}\left(  \mathbb{T}\right)  $,%
\[\left\Vert W\left(  t\right)  \phi\right\Vert _{X_{b,s}%
^{T}}\leq C\left\Vert \phi\right\Vert _{H^{s}\left(  \mathbb{T}\right)
}\text{;}\]

(ii) For any $f\in X_{b,s}^{T}$,
\[\left\Vert \int_{0}^{t}W\left(  t-\tau\right)  f\left(  \tau\right)
d\tau\right\Vert _{X_{b,s}^{T}}\leq C\left\Vert
f\right\Vert _{X_{b-1,s}^{T}}\text{.}\]
\end{lemma}


\begin{lemma}\label{lemma2}
	Let $b\in[0,1]$ and $u\in X_{b,s}^T$, then the function
	\begin{equation*}
	\begin{split}
	f:(0,T)&\rightarrow \R\\
	t&\mapsto \|u\|_{X_{b,s}^{T}}
	\end{split}
	\end{equation*}
	is continuous. Moreover, if $b>\frac12$, there exists $C:=C(b)>0$ such that
	\begin{equation}
	\lim_{t\rightarrow 0}f(t)\leq C\|u(0)\|_{H^s(\R)}.
	\end{equation}
\end{lemma}

\begin{lemma}\label{lemma4}
Let $b\in[0,1]$. If $\cup_{k=1}^{n} (a_k,b_k)$ is a finite covering of $(0,1)$, then there exists a constant $C>0$ depending only of the covering such that for every $u\in X_{b,s}$, we have
\begin{equation*}
\|u\|_{X_{b,s}[0,1]}\leq C\sum \|u\|_{X_{b,s}[a_k,b_k]}.
\end{equation*}
\end{lemma}


\begin{lemma}\label{integral}Let $s\in \mathbb{R}$.
	\begin{itemize}
		\item[(i)] For any $b\in \mathbb{R}$, follows that
		\begin{equation*}
		\|\psi(t)e^{it(\partial_x^2-\partial_x^4)}\|_{X_{s,b}}\leq c \|\psi(t)\|_{H^b(\mathbb{R})} \|u_0\|_{H^s(\mathbb{T})}.
		\end{equation*}
		\item[(ii)] 	Let $\psi\in C_0^{\infty}(\R)$ and $b,\ b'$ satisfying $0<b'\leq\frac12\leq b$ and $b+b'\leq 1$. We have for $T\leq 1$ the following inequality
		\begin{equation*}
		\left\|\psi\left(\frac{t}{T}\right)\int_0^te^{i(t-t')(\partial_x^2-\partial_x^4)}F(t')dt'\right\|_{X_{b,s}}\leq C T^{1-b-b'}\|F\|_{X_{-b',s}}.
		\end{equation*}
	\end{itemize}
\end{lemma}

\section{Well-posedness for 4NLS\label{sec3}}

In this section we are interested in the existence of solution for 4NLS with source and damping term.  More precisely, the following result can be proved:
\begin{theorem}\label{LWP}
	Let $T>0$, $s\geq 0$, $\lambda\in \mathbb{R}$ and $\frac12\leq b<\frac{11}{16}$. Set $a \in C_0^{\infty}(\mathbb{T})$  and $\varphi\in C_0^{\infty}(\mathbb{R})$ taking real values. For every $g\in L^2([-T,T]; H^s(\mathbb{T}))$ and $u_0\in H^s(\mathbb{T})$, there exists a unique  solution $u\in X_{b,s}^T$ of the Cauchy problem 
	\begin{equation}\label{dampingsource}
	\begin{cases}
	i\partial_tu +\partial_x^2u-\partial_x^4u
+i\varphi^2(t)a^2(x)u	=\lambda |u|^2u+g,& (x,t)\in \mathbb{T}\times \mathbb{R},\\
	u_0(x)=u_0(x),& x\in \mathbb{T}.\end{cases}
	\end{equation}
Furthermore, the flow map 
\begin{equation*}
\begin{split}
	F:H^s(\mathbb{T})\times L^2([-T,T];H^s(\mathbb{T}))&\rightarrow X_{b,s}^T\\
(u_0,g)&\mapsto u 
\end{split}
\end{equation*}
is Lipschitz on every bounded subset. The same result is valid for $a\in L^{\infty}(\mathbb{T})$.
\end{theorem}

\begin{proof}
	
	 Initially, we notice that $g\in X_{-b',s}^T$, for $b'\geq 0$. We restrict ourself to positive times. The solution on $[-T,0]$ can be obtained similarly. 
	
	Define the integral operator by
	\begin{equation}\label{integraleq}
	\Lambda (u)(t)= e^{-it(\partial_x^2-\partial_x^4)}u_0-i\int_0^te^{-i(t-t')(\partial_x^2- \partial_x^4)}(\lambda|u|^2u-i\varphi^2(t)a^2u+g)dt'.
	\end{equation} 
we are interested in applying the fixed point argument on the space $X_{b,s}^T$. To do it, let $\psi\in C_0^{\infty}(\mathbb{R})$ such that $\psi(t)=1$ for $t\in[-1,1]$. By Lemmas \ref{integral} and \ref{trilinear_estimates},  for $0<b'\leq\frac12\leq b$ and $b+b'\leq 1$, we have that
	\begin{equation}\label{estimate1}
	\begin{split}
	\|\Lambda(u)\|_{X_{b,s}^T}&\leq C\|u_0\|_{H^s(\mathbb{T})}+CT^{1-b-b'}\|\lambda|u|^2u-i\varphi^2(t)a+g\|_{X_{-b',s}^T}\\
	&\leq C\|u_0\|_{H^s(\mathbb{T})}+CT^{1-b-b'}\left(\|\varphi^2a^2u\|_{X_{0,s}^T}+\||u|^2u\|_{X_{b',s}^T}+\|g\|_{X_{-b',s}^T}\right)\\
		&\leq C\|u_0\|_{H^s(\mathbb{T})}+CT^{1-b-b'}\|u\|_{X_{b,s}^T}\left(1+\|u\|_{X_{b,0}^T}^2\right)+CT^{1-b-b'}\|g\|_{X_{-b',s}^T}.
	\end{split}
	\end{equation}
	In the same way, we get that
	\begin{equation}\label{estimate2}
	\|\Lambda(u)-\Lambda(v)\|_{X_{b,s}^T}\leq CT^{1-b-b'}(1+\|u\|_{X_{b,s}^T}^2+\|v\|_{X_{b,s}^T}^2)\|u-v\|_{X_{b,s}^T}.
	\end{equation}
These estimates imply that if T is chosen small enough $\Lambda$ is a contraction on a suitable ball of $X_{b,s}^T$. 

Let us prove the uniqueness in the class $X_{b,s}^T$ for the integral equation \eqref{integraleq}. Set 
\begin{equation}
w(t)=e^{-it(\partial_x^2-\partial_x^4)}u_0-i\int_0^te^{-i(t-t')(\partial_x^2- \partial_x^4)}(\lambda|u|^2u-i\varphi^2(t)a+g)dt'.
\end{equation}
By using \eqref{trilinear_estimates} we have that $|u|^2u \in X_{-b',s}^{T}$, for any $b'$ satisfying $-\frac{5}{16}<b'<\frac12$. Moreover, follows that
\begin{equation}
\partial_t\left[\int_0^te^{-i\tau(\partial_x^2-\partial_x^4)}(-i a^2\varphi^2u+\lambda |u|^2u+g)(\tau)\right]d \tau=e^{-it(\partial_x^2-\partial_x^4)}(-i a^2\varphi^2u+\lambda |u|^2u+g)(\tau),
\end{equation}
in the distributional sense. This implies that $w$ solves the following equation
\begin{equation}
i\partial_tw+\partial_x^2w-\partial_x^4 w+ia^2\varphi^2 u =\lambda |u|^2u
+g.	\end{equation}
Therefore, it follows that $v(t)=e^{-it(\partial_x^2-\partial_x^4)}(u-w)$ is solution of $\partial_tv=0$ and $v(t)=0$, respectively. Thus, $v=0$ and $u$  is solution of the integral equation. 

Let us prove the propagation of regularity. Firstly, note that for $u_0\in H^s(\mathbb{T})$, with $s>0$, we have an existence of time $T$ for the solution $u$ of integral equation in $X_{b,0}^T$ and another solution $\tilde{u}$ on time $\tilde{T}$ in $X_{b,s}^{\tilde{T}}$. By uniqueness in $X_{b,0}^T$ booth solutions $u$ and $\tilde{u}$ are the same on $[0,\tilde{T}]$. Admitting that $\tilde{T}< T$, we get that the blow up of $\|u(t)\|_{H^s(\mathbb{T})}$, as $t\to \tilde{T}$, while $\|u\|_{L^2(\mathbb{T})}$ remains bounded on this interval. By using local existence in $L^2(\mathbb{T})$ and Lemma \ref{lemma4} we see that$\|u\|_{X_{b,0}^{\tilde{T}}}$ is finite. Thus, estimate \eqref{estimate1} on $[\tilde{T}-\epsilon,\tilde{T}]$, with $\epsilon$ small enough such that $$C\epsilon^{1-b-b'}(1+\|u\|_{X_{b,0}[\tilde{T}-\epsilon,\tilde{T}]}^2)<\frac12,$$ ensures that
\begin{equation*}
\|u\|_{X_{b,s}[\tilde{T}-\epsilon,\tilde{T}]}\leq C(\|u(T-\epsilon)\|_{H^s(\mathbb{T})}+\|g\|_{X_{-b',s}}).
\end{equation*}
Then, $u\in X_{b,s}^{\tilde{T}}$, contradicting the blow up of $\|u(t)\|_{H^s(\mathbb{T})}$ near $\tilde{T}$. 

The second step is to use $L^2(\mathbb{T})$ energy estimates to obtain global existence in $L^2(\mathbb{T})$ and consequently, by using the above argument, in $H^s(\mathbb{T})$. Multiplying \eqref{dampingsource} by $\overline{u}$ , taking imaginary part, integrating by parts and using Cauchy--Schwarz inequality, we get
\begin{equation}\label{mass}
\|u(t)\|_{L^2(\mathbb{T})}^2\leq \|u_0\|_{L^2(\mathbb{T})}^2+C\int_0^t\|u(\tau)\|_{L^2(\mathbb{T})}^2d\tau+C\|g\|_{L^2([-T,T];L^2(\mathbb{T}))}.
\end{equation}
By using Gronwall inequality, we have that
\begin{equation*}
\|u(t)\|_{L^2(\mathbb{T})}^2\leq C(\|u_0\|_{L^2(\mathbb{T})}^2+\|g\|_{L^2([-T,T];L^2(\mathbb{T}))})e^{C|t|}.
\end{equation*}
Thus, the norm $L^2(\mathbb{T})$ remains bounded and the solution $u$ is global in time. 

Lastly, we prove the continuity of flow. Let $\tilde{u}$ the solution of \eqref{dampingsource} with $\tilde{u}_0$ and $\tilde{g}$, instead $u_0$ and $g$. A slight modification of \eqref{estimate2} yields that
\begin{equation*}
\begin{split}
	\|u-\tilde{u}\|_{X_{b,s}^T}\leq C\|u_0-\tilde{u}_0\|_{H^s(\mathbb{T})}+C\|g-\tilde{g}\|_{X_{-b',s}} +CT^{1-b-b'}(1+\|u\|_{X_{b,s}^T}^2+\|\tilde{u}\|_{X_{b,s}^T}^2)\|u-v\|_{X_{b,s}^T}.
	\end{split}
\end{equation*}
Then, for $T$ small enough depending on the size of $u_0$, $\tilde{u}_0$, $g$ and $\tilde{g}$, follows that
\begin{equation*}
\|u-\tilde{u}\|_{X_{b,s}^T}\leq C \|u_0-\tilde{u}_0\|_{H^s(\mathbb{T})}+C\|g-\tilde{g}\|_{X_{-b',s}^T}.
\end{equation*}
Thus, that the map data to solution is Lipschitz continuous on bounded sets for arbitrary $T$ and, consequently, the proof is complete. 
	\end{proof}
	
In the next two propositions we will give estimates that connect the solution $u$ of the 4NLS \eqref{dampingsource} with the damping term and source term.
\begin{proposition}\label{proposition2}
	For every $T>0$, $\eta>0$
 and $s\geq 0$, there exists $C(T, \eta,s)$ such that for every $u\in X_{b,s}^T$ solution of \eqref{dampingsource} with $\|u_0\|_{L^2(\mathbb{T})}+\|g\|_{L^2([0,T];H^s(\mathbb{T}))}\leq \eta$ the following estimate holds
 \begin{equation*}
 \|u\|_{X_{b,s}^T}\leq C(T,\eta,s)(\|u_0\|_{H^s(\mathbb{T})}+\|g\|_{L^2([0,T];H^s(\mathbb{T}))})
 \end{equation*}
	\end{proposition}
\begin{proof}
	Initially, assume $T\leq1$. Using \eqref{estimate1} we have that
	\begin{equation*}
	\|u\|_{X_{s,b}^T}	\leq C(\|u_0\|_{H^s(\mathbb{T})}+\|g\|_{X_{-b',s}^T})+C_ST^{1-b-b'}\|u\|_{X_{b,s}^T}\left(1+\|u\|_{X_{b,0}^T}^2\right).
	\end{equation*}
Choose $T$ such that $C_sT^{1-b-b'}\leq \frac12$, then the following inequality holds
	\begin{equation}\label{tec0}
	\|u\|_{X_{s,b}^T}	\leq C(\|u_0\|_{H^s(\mathbb{T})}+\|g\|_{X_{-b',s}^T})+C_sT^{1-b-b'}\|u\|_{X_{b,s}^T}\|u\|_{X_{b,0}^T}^2,
	\end{equation}
By using estimate \eqref{tec0}, for $s=0$, and choosing $T_1$ satisfying
	\begin{equation}\label{firstT}
	T_1^{1-b-b'}<\frac{1}{2C_0\left(\|u_0\|_{L^2(\mathbb{R})}+\|g\|_{X_{-b',0}^{T_1}}\right)^2},
	\end{equation}
	we obtain
	\begin{equation}\label{tec2}
	\|u\|_{X_{0,b}^{T_1}}	\leq C(\|u_0\|_{L^2(\mathbb{T})}+\|g\|_{X_{-b',0}^{T_1}})\leq C(\|u_0\|_{L^2(\mathbb{T})}+\|g\|_{L^2((T_1,T_1+\epsilon);L^2(\mathbb{T}))}).
	\end{equation}
	
	For the other hand, estimate \eqref{mass} implies 
	\begin{equation}\label{tec3}
	\|u(t)\|_{L^2(\mathbb{T})}\leq C \eta e^{C|t|}\leq C \eta e^{C}\leq C(\eta),
	\end{equation}
	where we have used that $T\leq1$. Then, thanks to \eqref{tec2} and \eqref{tec3}, there is a constant $\epsilon=\epsilon(\eta)$ such that
	\begin{equation}\label{tec11}
	\|u(T_1^{-}+s)\|_{X_{b,0}^{\epsilon}}	\leq C(\|u(T_1)\|_{L^2(\mathbb{T})}+\|g(t)\|_{L^2((T_1,T_1+\epsilon);L^2(\mathbb{T}))})
	\end{equation}
	and follows that the estimate \eqref{tec2} is valid for some large interval $[0,T]$, with $T\leq 1$, for any constant $C$ depending of $\eta$.
	
	Now, we back to the case $s>0$. For $T_s$ satisfying $C_sT_s^{1-b-b'}\leq \frac12$, by using \eqref{tec1} and \eqref{tec0}, we obtain
		\begin{equation*}
	C_sT_s^{1-b-b'}\|u\|_{X_{b,0}^{T_s}}^2\leq C_s T_s^{1-b-b'}C(\eta)^2\eta^2.
	\end{equation*}
Thus, follows that for an appropriate $T\leq\epsilon(\eta,s)$ the last expression can be controlled by $1/2$, therefore,  the following inequality is also true
	\begin{equation*}
	\|u\|_{X_{s,b}^T}	\leq C(\|u_0\|_{H^s(\mathbb{T})}+\|g\|_{L^2((0,T);H^s(\mathbb{T}))}).
	\end{equation*}
	Again, piecing solutions together, we get the same result for large $T\leq1$ with $C$ depending only of $\eta$ and $s$. Finally, the assumption $T\leq 1$ is removed similarly with a final constant $C(s,\eta,T )$.
\end{proof}

\begin{proposition}\label{proposition3}
	For every $T>0$ and $\eta>0$, there exists a constant $C(T,\eta)$ such that for all $s\geq 1$, we can find $C(T,\eta,s)$ such that $u\in X_{b,s}^T$ solution of \eqref{dampingsource} with $$\|u_0\|_{L^2(\mathbb{T})}+\|g\|_{L^2([0,T];H^s(\mathbb{T}))}\leq \eta,$$ satisfies
	\begin{equation*}
	\begin{split}
	\|u\|_{X_{b,s}^T}\leq& C(T,\eta)(\|u_0\|_{H^s(\mathbb{T})}+\|g\|_{L^2([0,T];H^s(\mathbb{T}))})\\
	& +C(T,\eta,s)(\|u\|_{X_{b,s-1}^T}\|u\|_{X_{b,1}^T}\|u\|_{X_{b,0}^T}+\|u\|_{X_{b,s-1}^T}).
	\end{split}
	\end{equation*} 
\end{proposition}
	
	\begin{proof}
	Initially, we assume $T\leq1$. By using Lemma \ref{integral} we have a constant $C$, independent of $s$ such that
	\begin{equation*}
	\begin{split}
	\|u\|_{X_{b,s}^T}\leq& \ C (\|u_0\|_{H^s(\mathbb{T})}+\|g\|_{L^2([0,T]:H^s(\mathbb{T}))})\\
	& +CT^{1-b-b'}(\|a^2 \varphi^2u\|_{L^2([0,T];H^s(\mathbb{T}))}+\||u|^2u\|_{X_{b',s-1}^T}),
	\end{split}
	\end{equation*} 
 for $b,\ b'$ satisfying $0\leq b'\leq\frac12\leq b$, $b+b'\leq 1$.
 
\cite[Lemma A.1]{Laurent-esaim} and Lemma \ref{trilinear_estimates} give us constants $C$ and $C_s$ such that
\begin{equation}\label{teccc0}
\begin{split}
\|u\|_{X_{b,s}^T}\leq& \ C (\|u_0\|_{H^s(\mathbb{T})}+\|g\|_{L^2([0,T];H^s(\mathbb{T}))})\\
& +T^{1-b-b'}(C\|u\|_{X_{b,s}^T}+C_s\|u\|_{X_{b,s-1}^T})\\
& +T^{1-b-b'}(C\|u\|_{X_{b,0}^T}^2\|u\|_{X_{b,s}^T}+C_s\|u\|_{X_{b,s-1}^T}\|u\|_{X_{b,1}^T}\|u\|_{X_{b,0}^T}).
\end{split}
\end{equation} 
From Proposition \ref{proposition2} we have that
\begin{equation}\label{teccc1}
\|u\|_{X_{b,0}^T}\leq C(\eta,T)\left(\|u_0\|_{L^2(\mathbb{T})}+\|g\|_{L^2([0,T];H^s(\mathbb{T}))}\right)\leq C(\eta,T)\eta.
\end{equation}
For $T\leq1$ in the last inequality $C(\eta):=C(\eta,1)$.
 
 By putting \eqref{teccc1} into \eqref{teccc0}, for small enough $T$ (depending only $\eta$), we get 
 	\begin{equation*}
 \begin{split}
 \|u\|_{X_{b,s}^T}\leq& C(\|u_0\|_{H^s(\mathbb{T})}+\|g\|_{L^2([0,T];H^s(\mathbb{T}))})\\
 &+C(s)(\|u\|_{X_{b,s-1}^T}\|u\|_{X_{b,1}^T}\|u\|_{X_{b,0}^T}+\|u\|_{X_{b,s-1}^T}).
 \end{split}
 \end{equation*}
Then, the conclusion of Lemma follows by a bootstrap argument. 	\end{proof}

\section{Local controllability\label{sec4}}
This section is devoted to prove the local controllability near of the null trajectory of the 4NLS \eqref{Controlzeroa} by a perturbative argument near the one done by E. Zuazua in \cite{Zuazua}. Then, we will use the fixed point theorem of Picard to deduce our result from the linear control.

First of all, we know (see, for instance, \cite{zz,zheng}) that any nonempty set $\omega$  satisfies an \textit{observability inequality} in $L^2(\mathbb{T})$ in arbitrary small time $T>0$. This means that:

\vspace{0.2cm}
\textit{For any $a(x)\in C^{\infty}(\mathbb{T})$ and $\varphi(t)\in C^{\infty}_0(0,T)$ real valued such that $a\equiv1$ on $\omega$ and $\varphi\equiv1$ on $[T/3,2T/3]$, there exists $C>0$ such that }
\begin{equation}\label{OI_L}
\left\Vert \Psi_0\right\Vert^2_{L^2(\mathbb{T})}\leq C\int_0^T\left\Vert a(x)\varphi(t)e^{it(\partial^2_x-\partial^4_x)}\Psi_0\right\Vert^2_{L^2(\mathbb{T})}dt,
\end{equation}
\textit{for every $\Psi_0\in L^2(\mathbb{T})$.}
\vspace{0.2cm}

\noindent Exact controllability property of a control system is equivalent to the observability of its adjoint system using the Hilbert Uniqueness Method introduced by J.-L. Lions \cite{lions1}. Thus, observability inequality \eqref{OI_L} implies the exact controllability in $L^2(\mathbb{T}):=L^2$ for the linear equation associated to \eqref{Controlzeroa}. 

To be precise, let us follow \cite[Section 5]{dehman-gerard-lebeau} to construct an isomorphism of control
\begin{equation*}
\begin{split}
\mathcal{R}:\ &L^2\to L^2\\
&\Phi_0\to\mathcal{R}\Phi_0=\Psi_0
\end{split}
\end{equation*}
such that if $\Phi$ is solution of the adjoint system 
\begin{equation}\label{adj_lin}
	\begin{cases}
	i\partial_t\Phi +\partial_x^2\Phi-\partial_x^4\Phi=0,& (x,t)\in \mathbb{T}\times \mathbb{R},\\
	\Phi(x,0)=\Phi_0(x),& x\in \mathbb{T}\end{cases}
	\end{equation}
and $\Psi$ solution of
\begin{equation*}
	\begin{cases}
	i\partial_t\Psi +\partial_x^2\Psi-\partial_x^4\Psi=a^2(x)\varphi^2(t)\Phi,& (x,t)\in \mathbb{T}\times \mathbb{R},\\
	\Psi(x,T)=\Psi_T(x)=0,& x\in \mathbb{T},\end{cases}
	\end{equation*}
we get that $\Psi(x,0)=\Psi_0(x)$. First, notice that application $\mathcal{R}$ has the following property:

\begin{lemma}\label{iso}
For every $s\geq0$, $\mathcal{R}$ is an isomorphism of $H^s(\mathbb{T})$.
\end{lemma}
\begin{proof}
To get the result we need to prove that  $\mathcal{R}$ maps $H^s(\mathbb{T})$ into itself and  $\mathcal{R}\Phi_0\in H^s(\mathbb{T})$ implies $\Phi_0\in H^s(\mathbb{T})$, that is, $D^s\Phi_0\in L^2$, where $D^s$ is defined by \eqref{k10}. Is not difficult to see that $\mathcal{R}$ maps $H^s(\mathbb{T})$ into itself. Thus, we check the following:

\vspace{0.2cm}
\noindent\textit{Claim:} $\mathcal{R}\Phi_0\in H^s(\mathbb{T})$ implies $\Phi_0\in H^s(\mathbb{T})$.
\vspace{0.2cm}

The claim is equivalent to show that $D^s\Phi_0\in L^2$, $D^s$ is defined by \eqref{k10}. Remember that, $$\mathcal{R}\Phi_0=i\int_0^Te^{-it(\partial_x^2-\partial_x^4)}\varphi^2a^2e^{it(\partial_x^2-\partial_x^4)}\Phi_0dt.$$
Since $\mathcal{R}^{-1}$ is continuous from $L^2$ into itself we get, using \cite[Lemma A.1]{Laurent-esaim}, that
\begin{align*}
\left\Vert D^s\Phi_0\right\Vert _{L^2}  &  \leq C\left\Vert \mathcal{R}D^s\Phi_0\right\Vert _{L^2}\leq C\left\Vert \int_0^Te^{-it(\partial_x^2-\partial_x^4)}\varphi^2a^2e^{it(\partial_x^2-\partial_x^4)}D^s\Phi_0dt\right\Vert _{L^2}\\
&  \leq C\left\Vert D^s \int_0^Te^{-it(\partial_x^2-\partial_x^4)}\varphi^2a^2e^{it(\partial_x^2-\partial_x^4)}\Phi_0dt\right\Vert _{L^2}\\&+C\left\Vert \int_0^Te^{-it(\partial_x^2-\partial_x^4)}[a^2,D^s]\varphi^2e^{it(\partial_x^2-\partial_x^4)}\Phi_0dt\right\Vert _{L^2}\\
&  \leq C\left\Vert \mathcal{R}\Phi_0\right\Vert _{H^s(\mathbb{T})}+C_sC\left\Vert \Phi_0\right\Vert _{H^{s-1}(\mathbb{T})}. 
\end{align*}
Thus, the result for $s\in[0,1]$ is proved. The result for $s\geq0$ can be guaranteed by iteration. Finally, above computation, for $s\geq 1$ we have
\begin{equation}\label{Iso_1}
\left\Vert \mathcal{R}^{-1}\Psi_0\right\Vert _{H^s(\mathbb{T})}\leq C(a,\psi,T)\left\Vert \Psi_0\right\Vert _{H^s(\mathbb{T})}+C(a,\psi,s,T)\left\Vert \Psi_0\right\Vert _{H^{s-1}(\mathbb{T})}.
\end{equation}
This complete the proof of claim and, consequently, lemma is verified.
\end{proof}

\subsection{Proof of Theorem \ref{main1}} Pick $a(x)\in C^{\infty}_0(\omega)$ and $\psi(t)\in C_0^{\infty}(0,T)$ different from zero, such that, observability inequality \eqref{OI_L} holds. We look for the function $g$ of the form $\varphi^2(t)a^2(x)\Phi$, where $\Phi$ is solution of \eqref{adj_lin} as in linear control problem.

We are here interested in choosing an appropriate $\Phi_0$ such that we can recover the controllability properties of the system \eqref{Controlzeroa}. Consider the two systems
\begin{equation*}
	\begin{cases}
	i\partial_t\Phi +\partial_x^2\Phi-\partial_x^4\Phi=0,& (x,t)\in \mathbb{T}\times \mathbb{R},\\
	\Phi(x,0)=\Phi_0(x),& x\in \mathbb{T}\end{cases}
	\end{equation*}
and 
\begin{equation*}
	\begin{cases}
	i\partial_tu+\partial_x^2u-\partial_x^4u=\lambda |u|^2u+a^2\varphi^2\Phi,& (x,t)\in \mathbb{T}\times \mathbb{R},\\
	u(x,T)=0,& x\in \mathbb{T}.\end{cases}
	\end{equation*}
We also define the operator $\mathcal{L}$ as follows
\begin{equation}
\begin{split}
	\mathcal{L}:L^2(\mathbb{T})&\rightarrow  L^2(\mathbb{T})\\
\Phi_0&\mapsto \mathcal{L}\Phi_0=u_0=u(0). 
\end{split}
\end{equation}
The goal is then to show that $\mathcal{L}$ is onto on a small neighborhood of the origin of $H^s$, for $s\geq 0$.  Split $u$ as $u=v+\Psi$, with $\Psi$ solution of
\begin{equation}
	\begin{cases}
	i\partial_t\Psi +\partial_x^2\Psi-\partial_x^4\Psi=a^2(x)\varphi^2(t)\Phi,& (x,t)\in \mathbb{T}\times \mathbb{R},\\
	\Psi(x,T)=0,& x\in \mathbb{T}.\end{cases}
	\end{equation}
	It corresponds to the linear control, and thus, $\Psi(0)=\mathcal{R}\Phi_0$. Moreover, observe that $v$ is solution of the system
	\begin{equation}\label{5.10}
	\begin{cases}
	i\partial_tv +\partial_x^2v-\partial_x^4v=\lambda |u|^2u,& (x,t)\in \mathbb{T}\times \mathbb{R},\\
	v(x,T)=0,& x\in \mathbb{T}.\end{cases}
	\end{equation} 
	Therefore, $u,v$ and $\Psi$ belong to $X^T_{b,0}$ and $u(0)=v(0)+\Psi(0)$, which we can write $\mathcal{L}\Phi_0$ as follows $$\mathcal{L}\Phi_0=\mathcal{K}\Phi_0+\mathcal{R}\Phi_0,$$where $\mathcal{K}\Phi_0=v(0)$. Observe that $\mathcal{L}\Phi_0=u_0$, equivalently,  $\Phi_0=-\mathcal{R}^{-1}\mathcal{K}\Phi_0+\mathcal{R}^{-1}u_0$.
	
	Define the operator 
	\begin{equation*}
	\begin{split}
	\mathcal{B}:\ &L^2\to L^2\\
	&\Phi_0\to \mathcal{B}\Phi_0=\mathcal{R}^{-1}\mathcal{K}\Phi_0+\mathcal{R}^{-1}u_0.
	\end{split}
	\end{equation*}
We want to prove that $\mathcal{B}$ has a fixed point. To do it, let us firstly define the following set $$F:=B_{L^2(\mathbb{T})}(0,\eta)\cap\left(\bigcap_{i=1}^{[s]-1}B_{H^i(\mathbb{T})}(0,R_i))\right)\cap B_{H^s(\mathbb{T})}(0,R_s),$$
for $\eta$ small enough and for some large $R_i$. We may assume $T<1$ and fix it, moreover, we will denote $C$ and $C_s=C(s)$ any constant depending only $a,\varphi,b,b', T$ and $s$, respectively.

By using Lemma \ref{iso} we have that $\mathcal{R}$ is an isomorphism of $H^s(\mathbb{T})$, thus 
\begin{align}\label{5.11}
\left\Vert \mathcal{B}\Phi_0\right\Vert _{H^s(\mathbb{T})}  &  \leq C_s\left(\left\Vert \mathcal{K}\Phi_0\right\Vert _{H^s(\mathbb{T})}+\left\Vert u_0\right\Vert _{H^s(\mathbb{T})}\right).
\end{align}
By the last inequality we should estimate $\left\Vert \mathcal{K}\Phi_0\right\Vert _{H^s(\mathbb{T})}=\left\Vert v_0\right\Vert _{H^s(\mathbb{T})}$. Then, for this, we will apply to equation \eqref{5.10} the same $X^T_{b,s}$ estimates which we used in the Theorem \ref{LWP}, more precisely, Lemma \ref{integral} and Lemma \ref{trilinear_estimates}, thus we get 
\begin{equation}\label{5.12}
\begin{split}
\left\Vert v(0)\right\Vert _{H^s(\mathbb{T})}  &  \leq C\left\Vert v\right\Vert _{X    ^T_{b,s}} \leq CT^{1-b-b'}\left\Vert \left\vert u \right\vert^2 u\right\Vert _{X^{T}_{-b',s}}\leq C\left\Vert \left\vert u \right\vert^2 u\right\Vert _{X^{T}_{-b',s}}\\&  \leq C_s\left\Vert u \right\Vert^2 _{X^T_{b,0}}\left\Vert u \right\Vert _{X^T_{b,s}}. 
\end{split}
\end{equation}
By the local linear behavior of $u$, that is, using Proposition \ref{proposition2}, we obtain that $$\left\Vert u \right\Vert _{X^T_{b,0}}\leq C\left\Vert \Phi_0\right\Vert _{L^2(\mathbb{T})},$$
for $$\left\Vert \varphi^2a^2\Phi \right\Vert _{L^2([0,T];L^2(\mathbb{T}))}\leq C\left\Vert \Phi_0 \right\Vert _{L^2(\mathbb{T})}<C\eta<1.$$
Finally applying \eqref{5.11} and \eqref{5.12}, with $s=0$, this ensures that
$$
\left\Vert \mathcal{B}\Phi_0\right\Vert _{L^2(\mathbb{T})}  \leq C\left(\left\Vert \Phi_0\right\Vert^3 _{L^2(\mathbb{T})}+\left\Vert u_0\right\Vert _{L^2(\mathbb{T})}\right).
$$
Then, by the last inequality, choosing $\eta$ small enough and $\left\Vert u_0\right\Vert _{L^2(\mathbb{T})}\leq\frac{\eta}{2C}$, we have that $$\left\Vert \mathcal{B}\Phi_0\right\Vert _{L^2(\mathbb{T})}  \leq \eta$$ and, therefore, $\mathcal{B}$ reproduces the ball $B_{\eta}$ in $L^2(\mathbb{T})$. 

To prove the result on a small neighborhood of the origin $H^s(\mathbb{T})$, we will divide the proof in two steps.

\vspace{0.2cm}
\noindent
\textit{\textbf{Step 1.}} $s\in (0,1]$
\vspace{0.2cm}

For $s\leq1$, we came back to \eqref{5.12} with the following new estimates in $X^T_{b,s}$
$$\left\Vert v(0)\right\Vert _{H^s(\mathbb{T})}  \leq C_s\eta^2\left\Vert u \right\Vert^2 _{X^T_{b,s}}$$
and
$$\left\Vert \mathcal{B}\Phi_0\right\Vert _{H^s(\mathbb{T})}  \leq C_s\left(\eta^2\left\Vert u\right\Vert _{X^T_{b,s}}+\left\Vert u_0\right\Vert _{H^s(\mathbb{T})}\right).$$
Thus, using the Proposition \ref{proposition2} for $$\left\Vert \varphi^2a^2\Phi \right\Vert _{L^2([0,T];L^2(\mathbb{T}))}\leq C\left\Vert \Phi_0 \right\Vert _{L^2(\mathbb{T})}<C\eta<1,$$ we have
\begin{equation}\label{5.13}
\left\Vert u\right\Vert _{X^T_{b,s}}  \leq C_s\left\Vert \Phi_0 \right\Vert^2 _{X^T_{b,s}}
\end{equation}
and
$$\left\Vert \mathcal{B}\Phi_0\right\Vert _{H^s(\mathbb{T})}  \leq C_s\left(\eta^2\left\Vert \Phi_0\right\Vert _{H^s(\mathbb{T})}+\left\Vert u_0\right\Vert _{H^s(\mathbb{T})}\right).$$
Then, by this two inequalities, for $C_s\eta^2<1/2$, $\mathcal{B}$ reproduces any ball in $H^s(\mathbb{T})$ of the radius greater than $2C_s\left\Vert u_0\right\Vert _{H^s(\mathbb{T})}$. Therefore, we conclude that $\mathcal{B}$ reproduces the ball in $F$, if $\eta<\tilde{C}_s$, $\left\Vert u_0\right\Vert _{H^s(\mathbb{T})}\leq C(\eta)$ and $R\geq C(\left\Vert u_0\right\Vert _{H^s(\mathbb{T})})$. Furthermore, since the estimates are uniform in $s\in(0,1]$ the bound on $\eta$ is also uniform.

\vspace{0.2cm}
\noindent
\textit{\textbf{Step 2.}} $s>1$
\vspace{0.2cm}

We beginning choosing $R_i$ by induction as follows: We chosen $R_1$ as the previous case so that $\mathcal{B}$ reproduces $B_{H^1(\mathbb{T})}(0,R_1)$. Is important, in this point, to make some assumptions of smallness on $\eta$ which on will be independent of $i$ and $s$. 
Firstly, using the estimate \eqref{Iso_1} we get 
\begin{equation*}
\left\Vert \mathcal{B}\Phi_0\right\Vert _{H^i(\mathbb{T})}\leq C\left\Vert \mathcal{K}\Phi_0\right\Vert _{H^i(\mathbb{T})}+C_i\left\Vert \mathcal{K}\Phi_0\right\Vert _{H^{i-1}(\mathbb{T})}+C_i\left\Vert u_0\right\Vert _{H^{i}(\mathbb{T})}.
\end{equation*}
Analogously, for $s\in(0,1]$, we have that
\begin{equation*}
\left\Vert \mathcal{K}\Phi_0\right\Vert _{H^{i-1}(\mathbb{T})}\leq C_{i-1}\eta^2\left\Vert \Phi_0\right\Vert _{H^{i-1}(\mathbb{T})}\leq C_{i-1}\eta^2R_{i-1}.
\end{equation*}
Using multilinear estimate, Lemma \ref{trilinear_estimates}, the following holds
\begin{align*}
\left\Vert v(0)\right\Vert _{H^i(\mathbb{T})}  &  \leq C\left\Vert v\right\Vert _{X    ^T_{b,i}} \leq C\left\Vert \left\vert u \right\vert^2 u\right\Vert _{X^{T}_{-b',i}} \leq C\left\Vert u \right\Vert^2 _{X^T_{b,0}}\left\Vert u \right\Vert _{X^T_{b,i}}+C_i\left\Vert u \right\Vert_{X^T_{b,i-1}}\left\Vert u \right\Vert _{X^T_{b,1}}\left\Vert u \right\Vert _{X^T_{b,0}} .
\end{align*}
Right now, we would like to bound the term with maximum derivative. For this, we use Proposition \ref{proposition3} and \cite[Corollary A.2]{Laurent-esaim} to obtain
\begin{align*}
\left\Vert u\right\Vert _{X^T_{b,1}(\mathbb{T})} &  \leq C\left\Vert \varphi^2a^2\Phi \right\Vert_{L^2([0,T];H^i(\mathbb{T}))}+C_i\left\Vert u \right\Vert _{X^T_{b,i-1}}+C_i\left\Vert u \right\Vert _{X^T_{b,i-1}}\left\Vert u \right\Vert _{X^T_{b,1}}\left\Vert u \right\Vert _{X^T_{b,0}} \\&  \leq C\left\Vert \Phi_0 \right\Vert_{H^i(\mathbb{T})}+C_i\left\Vert \Phi_0 \right\Vert_{H^{i-1}(\mathbb{T})}+C_i\left\Vert u \right\Vert _{X^T_{b,i-1}}+C_i\left\Vert u \right\Vert _{X^T_{b,i-1}}\left\Vert u \right\Vert _{X^T_{b,1}}\left\Vert u \right\Vert _{X^T_{b,0}}.
\end{align*}
By using \eqref{5.13}, we can also bound the lower derivative, which yields
\begin{align*}
\left\Vert v(0)\right\Vert _{H^i(\mathbb{T})}  &  \leq C\eta^2\left\Vert u \right\Vert^2 _{X^T_{b,i}}+C_iR_{i-1}R_1\eta \\&  \leq C\eta^2\left\Vert \Phi_0 \right\Vert^2 _{H^i(\mathbb{T})}+ C\eta^2(C_iR_{i-1}+C_iR_{i-1}R_1\eta)+C_iR_{i-1}R_1\eta.
\end{align*}
Finally, we ensures that $$\left\Vert \mathcal{B}\Phi_0\right\Vert _{H^i(\mathbb{T})}\leq C\eta^2\left\Vert \Phi_0\right\Vert _{H^i(\mathbb{T})}+C (i,\eta,R_1,R_{i-1},\left\Vert u_0\right\Vert _{H^{i}(\mathbb{T})}).$$
Choosing $C\eta^2<1/2$ independent of $s$ and $R_i=C (i,\eta,R_1,R_{i-1},\left\Vert u_0\right\Vert _{H^{i}(\mathbb{T})})$, then $B$ reproduces $B_{H^i}(0,R_i)$. The same argument is still valid for $s\geq1$ and Step 2 is thus proved.

\vspace{0.2cm}

To finalize, $\mathcal{B}$ is contracting for $L^2(\mathbb{T})-$norm. Indeed, consider the following systems
\begin{equation}\label{adj_lin_2a}
	\begin{cases}
	i\partial_t(u-\tilde{u})+\partial_x^2(u-\tilde{u})-\partial_x^4(u-\tilde{u})=\lambda (|u|^2u-|\tilde{u}|^2\tilde{u})+a^2\varphi^2(\Phi-\tilde{\Phi}),& (x,t)\in \mathbb{T}\times \mathbb{R},\\
	(u-\tilde{u})(x,T)=0,& x\in \mathbb{T}\end{cases}
	\end{equation}
	and 
	\begin{equation*}
	\begin{cases}
	i\partial_t(v-\tilde{v}) +\partial_x^2(v-\tilde{v})-\partial_x^4(v-\tilde{v})=\lambda(|u|^2u-|\tilde{u}|^2\tilde{u}),& (x,t)\in \mathbb{T}\times \mathbb{R},\\
	(v-\tilde{v})(x,T)=0,& x\in \mathbb{T}.\end{cases}
	\end{equation*} 
Using Lemma \ref{trilinear_estimates}, follows that
\begin{equation}\label{beta}
\begin{split}
\left\Vert \mathcal{B}\Phi_0-\mathcal{B}\tilde{\Phi}_0\right\Vert _{L^2(\mathbb{T})}& \leq \left\Vert (v-\tilde{v})(0)\right\Vert _{L^2(\mathbb{T})}  \leq C\left\Vert (v-\tilde{v})\right\Vert _{X    ^T_{b,0}}\\
&\leq C T^{1-b-b'}\left\Vert \left\vert u \right\vert^2 u - \left\vert \tilde{u} \right\vert^2 \tilde{u}\right\Vert _{X^{T}_{-b',0}}\\
&  \leq C\left(\left\Vert u \right\Vert^2 _{X^T_{b,0}}+\left\Vert \tilde{u} \right\Vert^2 _{X^T_{b,0}} \right)\left\Vert u -\tilde{u}\right\Vert _{X^T_{b,0}}\\& \leq C\eta^2\left\Vert u -\tilde{u}\right\Vert _{X^T_{b,0}}.
\end{split}
\end{equation}
To bound $\left\Vert u -\tilde{u}\right\Vert _{X^T_{b,0}}$ in the last inequality \eqref{beta}, we use the equation \eqref{adj_lin_2a} to deduce
\begin{equation*}
\begin{split}
\left\Vert u-\tilde{u}\right\Vert _{X^T_{b,0}(\mathbb{T})} &  \leq C\left\Vert \varphi^2a^2(\Phi-\tilde{\Phi}) \right\Vert_{L^2([0,T];L^2(\mathbb{T}))}+C T^{1-b-b'}\left\Vert \left\vert u \right\vert^2 u - \left\vert \tilde{u} \right\vert^2 \tilde{u}\right\Vert _{X^{T}_{-b',0}} \\&  \leq C\left\Vert \Phi_0-\tilde{\Phi}_0 \right\Vert_{L^2(\mathbb{T})}+C\left(\left\Vert u \right\Vert^2 _{X^T_{b,0}}+\left\Vert \tilde{u} \right\Vert^2 _{X^T_{b,0}} \right)\left\Vert u -\tilde{u}\right\Vert _{X^T_{b,0}}\\&\leq C\left\Vert \Phi_0-\tilde{\Phi}_0 \right\Vert_{L^2(\mathbb{T})}+C\eta^2\left\Vert u -\tilde{u}\right\Vert _{X^T_{b,0}}.
\end{split}
\end{equation*}
Taking $\eta$ small enough (independent on $s$) it yields
\begin{equation}\label{phi}
\left\Vert u-\tilde{u}\right\Vert _{X^T_{b,0}(\mathbb{T})}\leq C\left\Vert \Phi_0-\tilde{\Phi}_0 \right\Vert_{L^2(\mathbb{T})}.
\end{equation}
To finish, combining \eqref{phi} into \eqref{beta} follows that
$$\left\Vert \mathcal{B}\Phi_0-\mathcal{B}\tilde{\Phi}_0\right\Vert _{L^2(\mathbb{T})}\leq C\eta^2\left\Vert \Phi_0-\tilde{\Phi}_0 \right\Vert_{L^2(\mathbb{T})}.$$
Therefore, $\mathcal{B}$ is a contraction of a closed set $F$ of $L^2(\mathbb{T}) $, for $\eta$ small enough (independent on $s$). In addition, $\mathcal{B}$ has a fixed point which is, by construction, belongs to $H^s(\mathbb{T})$. This archived the proof of Theorem \ref{main1}. \qed

\section{Propagation of compactness and regularity in Bourgain spaces\label{sec5}}

We present, in this section, some properties of propagation in Bourgain spaces for the linear differential operator $L=i\partial_{t}+\partial_{x}^{2}-\partial_{x}^{4}$ associated with the fourth order Schr\"odinger equation.
We will adapt  the results due Dehman-G\'{e}rard-Lebeau
\cite[Propositions 13 and 15]{dehman-gerard-lebeau}, in the case of $X_{b,s}$ spaces, of the Schr\"odinger operator. These results of propagation are the key to prove the global stabilization. The main ingredient is basically pseudo-differential analysis. Let us begin with a result of propagation of compactness which will ensure strong convergence in appropriate spaces for the study of the global stabilization.
\begin{proposition}
[\textit{Propagation of compactness}]\label{prop_a}Let $T>0$ and $0\leq b^{\prime}\leq
b\leq1$ be given. Suppose that $u_{n}\in X_{b,0}^{T}$ and
$f_{n}\in X_{-b,-3+3b}^{T}$ satisfying%
\[
i\partial_{t}u_{n}+\partial_{x}^{2}u_{n}-\partial_{x}^{4}u_{n}=f_{n}\text{,}%
\]
for $n=1,2,3,\ldots$. Assume that there exists a constant $C>0$ such that
\begin{equation}\label{k14}
\left\Vert u_{n}\right\Vert _{X_{b,0}^{T}}\leq C
\end{equation}
and
\begin{equation}
\label{k15}
\left\Vert u_{n}\right\Vert _{X_{-b,-3+3b}^{T}}+\left\Vert f_{n}\right\Vert
_{X_{-b,-3+3b}^{T}}+\left\Vert u_{n}\right\Vert _{X_{-b^{\prime}%
,-1+3b^{\prime}}^{T}}\rightarrow0\text{, as }n\rightarrow+\infty\text{.}
\end{equation}
In addition, assume that for some nonempty open set $\omega\subset\mathbb{T}$%
\[
u_{n}\rightarrow0\text{ strongly in }L^{2}\left(  0,T;L^{2}\left(
\omega\right)  \right)  \text{.}%
\]
Then
\[
u_{n}\rightarrow0\text{ strongly in }L_{loc}^{2}\left([0,T]
;L^{2}\left(  \mathbb{T}\right)  \right)  \text{.}%
\]

\end{proposition}

\begin{proof}
Pick $\varphi\in C^{\infty}\left(  \mathbb{T}\right)  $ and $\psi\in
C_{0}^{\infty}\left(  0,T \right)  $ real valued and set%
\[
B=\varphi\left(  x\right)  D^{-3}\text{ \ and \ }A=\psi\left(  t\right)
B,
\]
where $D^{-3}$ is defined by \eqref{k10}. Then%
\[
A^{\ast}=\psi\left(  t\right)  D^{-3}\varphi\left(  x\right)  \text{.}%
\]
For $\epsilon>0$, we denote $A_{\epsilon}=Ae^{\epsilon\partial_{x}
^{2}}  =\psi\left(  t\right)  B_{\epsilon}$ for the regularization of $A$. By a classical way, we can write
\begin{align*}
\alpha_{n,\epsilon} &  =i\left(  -\psi^{\prime}\left(  t\right)  B_{\epsilon}u_{n},u_{n}\right)
+\left(  A_{\epsilon}u_{n},\left( \partial_{x}^{2}-\partial_{x}^{4}\right)  u_{n}\right) \\
&  =\left(  \left[  A_{\epsilon},\partial_{x}^{2}-\partial_{x}^{4}\right]  u_{n},u_{n}\right)  -i\left(  \psi^{\prime}\left(
t\right)  B_{\epsilon}u_{n},u_{n}\right)  \text{.}%
\end{align*}
On the other hand, we have
\begin{equation}\label{k16_a}
\alpha_{n,\epsilon}=\left(  f_{n},A_{\epsilon}^{\ast}u_{n}\right)
_{L^{2}\left(  \mathbb{T\times}\left(  0,T\right)  \right)  }-\left(
A_{\epsilon}u_{n},f_{n}\right)  _{L^{2}\left(  \mathbb{T\times}\left(
0,T\right)  \right)  }\text{.} 
\end{equation}
By using H\"{o}lder inequality and Lemma \ref{estimates_a}, we get that
\begin{align*}
\left\vert \left(  f_{n},A_{\epsilon}^{\ast}u_{n}\right)  _{L^{2}\left(
\mathbb{T\times}\left(  0,T\right)  \right)  }\right\vert  &  \leq\left\Vert
f_{n}\right\Vert _{X_{-b,-3+3b}^{T}}\left\Vert A_{\epsilon}^{\ast}u_{n}\right\Vert
_{X_{b,3-3b}^{T}}\\
&  \leq\left\Vert f_{n}\right\Vert _{X_{-b,-3+3b}^{T}}\left\Vert
u_{n}\right\Vert _{X_{b,0}^{T}}.
\end{align*}
Therefore, from \eqref{k14} and \eqref{k15}, follows that
\begin{equation}\label{k16}
\lim_{n\rightarrow\infty}\sup_{0<\epsilon\leq1}\left\vert \left(
f_{n},A_{\epsilon}^{\ast}u_{n}\right)  _{L^{2}\left(  \mathbb{T\times}\left(
0,T\right)  \right)  }\right\vert =0\text{.} 
\end{equation}
Similar computations yields that
\begin{equation}\label{k18}
\begin{split}
\lim_{n\rightarrow\infty}\sup_{0<\epsilon\leq1}\left\vert \left(  A_{\epsilon
}u_{n},f_{n}\right)  _{L^{2}\left(  \mathbb{T\times}\left(  0,T\right)
\right)  }\right\vert &=0 
\\
\lim_{n\rightarrow\infty}\sup_{0<\epsilon\leq1}\left\vert \left(  \psi
^{\prime}\left(  t\right)  B_{\epsilon}u_{n},u_{n}\right)  _{L^{2}\left(
\mathbb{T\times}\left(  0,T\right)  \right)  }\right\vert &=0\text{.}
\end{split}
\end{equation}
Thus, thanks to \eqref{k16_a}-\eqref{k18}, the following holds%
\[
\lim_{n\rightarrow\infty}\sup_{0<\epsilon\leq1}\left\vert \alpha_{n,\epsilon
}\right\vert =0
\]
and, therefore,
\[
\lim_{n\rightarrow\infty}\sup_{0<\epsilon\leq1}\left\vert \left(  \left[
A_{\epsilon},\partial_{x}^{2}-\partial_{x}^{4}\right]
u_{n},u_{n}\right)  _{L^{2}\left(  \mathbb{T\times}\left(  0,T\right)
\right)  }\right\vert =0,
\]
particularly,
\[
\lim_{n\rightarrow\infty}\left\vert \left(  \left[  A,\partial_{x}^{2}-\partial_{x}^{4}\right]  u_{n},u_{n}\right)
_{L^{2}\left(  \mathbb{T\times}\left(  0,T\right)  \right)  }\right\vert
=0\text{.}%
\]
As $D_x^{-3}:=D^{-3}$ commutes with $\partial_{x}^{k}$, for $k=1,2,3$, we have that%
\begin{equation}\label{k19}
\begin{split}
\left[  A,\partial_{x}^{2}-\partial_{x}^{4}\right]   
=&\left[  \psi\left(  t\right)  B,\partial_{x}^{2}-\partial_{x}^{4}\right] \\
  =&\left[  \psi\left(  t\right)  \varphi\left(  x\right)  D^{-3}%
,\partial_{x}^{2}-\partial_{x}^{4}\right] \\
 =& 4\psi\left(  t\right)  \left(  \partial_{x}^{3}\varphi\right)
\partial_{x}D^{-3}+12\psi\left(  t\right)  \left(  \partial_{x}^{2}%
\varphi\right)  \partial_{x}^{2}D^{-3}+4\psi\left(  t\right)  \left(  \partial_{x}\varphi\right)  \partial_{x}^{3}D^{-3}\\
&  -2\psi\left(  t\right)  \left(  \partial_{x}\varphi\right)
\partial_{x}D^{-3} -\psi\left(  t\right)  \left(  \partial_{x}^{2}\varphi-\partial_{x}^{4}\varphi\right)  D^{-3}=:\sum\limits_{i=1}^{5}I_{i}
\end{split}
\end{equation}
Note that, using \eqref{k14} and \eqref{k15}, we bounded $I_5$ by
\[
\left(  \psi\left(  t\right)  \left( \partial_{x}^{2}-\partial_{x}^{4}\right)  D^{-3}u_{n},u_{n}\right)
_{L^{2}\left(  \mathbb{T\times}\left(  0,T\right)  \right)  }\leq\left\Vert
u_{n}\right\Vert _{X_{-b,-3+3b}^{T}}\left\Vert u_{n}\right\Vert _{X_{b,0}^{T}%
}\text{.}%
\]
Indeed, 
\begin{align*}
\left(  \psi\left(  t\right)  \left( \partial_{x}^{2}-\partial_{x}^{4}\right)  D^{-3}u_{n},u_{n}\right)
_{L^{2}\left(  \mathbb{T\times}\left(  0,T\right)  \right)  }
&\leq C \left\Vert \psi\left(  t\right)  \left(  \partial_{x}^{2}-\partial_{x}^{4}\right)\varphi  \partial_{x}D^{-3}u_{n}\right\Vert _{X_{-b,0}^{T}}\left\Vert u_{n}\right\Vert _{X_{b,0}^{T}}\nonumber\\
&  \leq C\left\Vert (L\varphi)D^{-3}u_{n}\right\Vert _{X_{-b,0}^{T}}\left\Vert
u_{n}\right\Vert _{X_{b,0}^{T}}\nonumber\\
&  \leq C\left\Vert D^{-3}u_{n}\right\Vert _{X_{-b,3b}^{T}}\left\Vert
u_{n}\right\Vert _{X_{b,0}^{T}}\nonumber\\
&  \leq C\left\Vert u_{n}\right\Vert _{X_{-b,-3+3b}^{T}}\left\Vert
u_{n}\right\Vert _{X_{b,0}^{T}}\text{.}
\end{align*}%
Arguing as made in \eqref{k16}, we infer that
\[
\left(  \psi\left(  t\right)  \left( \partial_{x}^{2}-\partial_{x}^{4}\right)  D^{-3}u_{n},u_{n}\right)\rightarrow0\text{, \ as }n\rightarrow+\infty\text{.}%
\]
Note that for the terms $I_{i}$, $i=1,2$ and $4$ in (\ref{k19}), the
loss of regularity is too large if we use the estimate with the same $b$. Using
the index $b^{\prime}$ instead of $b$, we have%
\begin{equation}\label{k20}
\begin{split}
\left(  I_{1}u_{n},u_{n}\right)  _{L^{2}\left(  \mathbb{T\times}\left(
0,T\right)  \right)  }  &  =4\left(  \psi\left(  t\right)  \left(
\partial_{x}^{3}\varphi\right)  \partial_{x}D^{-3}u_{n},u_{n}\right)
_{L^{2}\left(  \mathbb{T\times}\left(  0,T\right)  \right)  }\\
&\leq C \left\Vert \psi\left(  t\right)  \left(  \partial_{x}^{3}%
\varphi\right)  \partial_{x}D^{-3}u_{n}\right\Vert _{X_{b^{\prime
},2-3b^{\prime}}^{T}}\left\Vert u_{n}\right\Vert _{X_{-b^{\prime
},-2+3b^{\prime}}^{T}}\\\
&  \leq C\left\Vert u_{n}\right\Vert _{X_{b^{\prime},0}^{T}}\left\Vert
u_{n}\right\Vert _{X_{-b^{\prime},-2+3b^{\prime}}^{T}}\\
\left(  I_{2}u_{n},u_{n}\right)  _{L^{2}\left(  \mathbb{T\times}\left(
0,T\right)  \right)  }  &  =12\left(  \psi\left(  t\right)  \left(
\partial_{x}^{2}\varphi\right)  \partial_{x}^{2}D^{-3}u_{n},u_{n}\right)
_{L^{2}\left(  \mathbb{T\times}\left(  0,T\right)  \right)  }\\
&  \leq C\left\Vert \psi\left(  t\right)  \left(  \partial_{x}^{2}
\varphi\right)  \partial_{x}^{2}D^{-3}u_{n}\right\Vert _{X_{b^{\prime
},1-3b^{\prime}}^{T}}\left\Vert u_{n}\right\Vert _{X_{-b^{\prime
},-1+3b^{\prime}}^{T}}\\
&  \leq C\left\Vert u_{n}\right\Vert _{X_{b^{\prime},0}^{T}}\left\Vert
u_{n}\right\Vert _{X_{-b^{\prime},-1+3b^{\prime}}^{T}} 
\end{split}
\end{equation}
and
\begin{equation}\label{k23}
\begin{split}
\left(  I_{4}u_{n},u_{n}\right)  _{L^{2}\left(  \mathbb{T\times}\left(
0,T\right)  \right)  }  &  =-2\left(  \psi\left(  t\right)  \left(
\partial_{x}\varphi\right)  \partial_{x}D^{-3}u_{n},u_{n}\right)
_{L^{2}\left(  \mathbb{T\times}\left(  0,T\right)  \right)  }\\
&  \leq C\left\Vert \psi\left(  t\right)  \left(  \partial_{x}\varphi\right)
\partial_{x}D^{-3}u_{n}\right\Vert _{X_{b^{\prime},2-3b^{\prime}}^{T}%
}\left\Vert u_{n}\right\Vert _{X_{-b^{\prime},-2+3b^{\prime}}^{T}}\\
&  \leq C\left\Vert u_{n}\right\Vert _{X_{b^{\prime},0}^{T}}\left\Vert
u_{n}\right\Vert _{X_{-b^{\prime},-2+3b^{\prime}}^{T}}\text{.} 
\end{split}
\end{equation}
Observe that 
\begin{equation}\label{k25}
X_{-b^{\prime},-1+3b^{\prime}}^{T}\hookrightarrow X_{-b^{\prime}%
,-2+3b^{\prime}}^{T}%
\end{equation}
where $\hookrightarrow$ denotes a compact imbedding. Thus, from \eqref{k14},
\eqref{k15} and \eqref{k25}, we have that \eqref{k20}-\eqref{k23} tends to $0$
as $n\rightarrow+\infty$. 

To conclude the proof remains analyze the third term of \eqref{k19}, that is, $I_3$. Remark that $-\partial_{x}^{3}D^{-3}$ is the
orthogonal projection on the subspace of functions with $\hat{u}\left(
0\right)  =0$. Futhermore,
\[
X_{b,0}^{T}\hookrightarrow X_{0,0}^{T}\hookrightarrow X_{-b^{\prime},0}%
^{T}\text{, for }0\leq b^{\prime}\leq b\leq1\text{,}%
\]
thus, using the Rellich Theorem, we see that%
\[
\hat{u}_{n}\left(  0,t\right)  \longrightarrow\hat{u}\left(  0,t\right)
=0\text{ in }X_{0,0}^{T}\equiv L^{2}\left(  0,T\right)  \text{ strongly,}%
\]
and hence%
\[
\left(  \psi\left(  t\right)  \left(  \partial_{x}\varphi\right)  \hat{u}%
_{n}\left(  0,t\right)  ,u_{n}\right)  _{L^{2}\left(  \mathbb{T\times}\left(
0,T\right)  \right)  }\longrightarrow0\text{.}%
\]
We have proved that, for any $\varphi\in C^{\infty}\left(  \mathbb{T}\right)
$ and $\psi\in C_{0}^{\infty}\left(  \left(  0,T\right)  \right)  $,%
\[
\left(  \psi\left(  t\right)  \left(  \partial_{x}\varphi\right) \partial^3_xD^{-3}u
_{n}  ,u_{n}\right)  _{L^{2}\left(  \mathbb{T\times}\left(
0,T\right)  \right)  }\longrightarrow0\text{.}%
\]

Observe that $\phi\in C^{\infty}\left(  \mathbb{T}\right)  $ can be
written in the form $\partial_{x}\varphi$ for some function $\varphi\in
C^{\infty}\left(  \mathbb{T}\right)  $ if and only if $\int_{\mathbb{T}}%
\phi\left(  x\right)  dx=0$. Thus, for any $\chi\in C_{0}^{\infty}\left(
\omega\right)  $ and any $x_{0}\in\mathbb{T}$, $\phi\left(  x\right)
=\chi\left(  x\right)  -\chi\left(  x-x_{0}\right)  $ can be written as
$\phi=\partial_{x}\varphi$ for some $\varphi\in C^{\infty}\left(
\mathbb{T}\right)  .$ Since $u_{n}$ is strongly convergent to $0$ in $L^{2}\left(  0,T;L^{2}\left(
\omega\right)  \right)  $,%
\[
\lim_{n\rightarrow\infty}\left(  \psi\left(  t\right)  \chi u_{n},u_{n}\right)
_{L^{2}\left(  \mathbb{T\times}\left(  0,T\right)  \right)  }=0\text{.}%
\]
Then, for any $x_{0}\in\mathbb{T}$,%
\[
\lim_{n\rightarrow\infty}\left(  \varphi\left(  t\right)  \chi\left(
\cdot-x_{0}\right)  u_{n},u_{n}\right)  _{L^{2}\left(  \mathbb{T\times}\left(
0,T\right)  \right)  }=0\text{.}%
\]
Finally, we closed the proof constructing a partion of unity on $\mathbb{T}$
with some functions of the form $\chi_{i}\left(  \cdot-x_{0}^{i}\right)  $,
with $\chi_{i}\in C_{0}^{\infty}\left(  \omega\right)  $ and $x_{0}^{i}%
\in\mathbb{T}$.
\end{proof}

To close this section we prove the gain of regularity of the linear fourth order Schr\"odinger equation.

\begin{proposition}
[\textit{Propagation of regularity}]\label{prop_b}Let $T>0$, $0\leq b<1$ and $f\in
X_{-b,r}^{T}$ be given. Let\ $u\in X_{b,r}^{T}$ be a solution of%
\[
i\partial_{t}u+\partial_{x}^{2}u-\partial_{x}^{4}u=f\text{.}%
\]
If there exists a nonempty $\omega\subset\mathbb{T}$ such that $u\in
L_{loc}^{2}\left([0,T];H^{r+\rho}\left(  \omega\right)
\right)  $ for some $\rho$ with%
\[
0<\rho\leq\min\left\{  \frac{3}{2}(1-b),\frac{1}{2}\right\}  \text{,}%
\]
then%
\[
u\in L_{loc}^{2}\left([0,T];H^{r+\rho}\left(  \mathbb{T}%
\right)  \right)  \text{.}%
\]

\end{proposition}

\begin{proof}
We first regularize $u_{n}=\exp\left(  \frac{1}{n}\partial_{x}^{2}\right)
u:=\Theta_{n}u$ and $f_{n}:=\Theta_{n}f$, with%
\begin{equation*}
\left\Vert u_{n}\right\Vert _{X_{b,r}^{T}}\leq C\text{ \ and \ }\left\Vert
f_{n}\right\Vert _{X_{-b,r}^{T}}\leq C\text{,} 
\end{equation*}
for some constant $C>0$ and $n=1,2,\dots$. 

Let $s=r+\rho$, $\varphi\in C^{\infty}\left(  \mathbb{T}\right)  $ and $\psi\in
C_{0}^{\infty}\left(  0,T\right) $ taking real values. Set
$Bu=D^{2s-3}\varphi\left(  x\right)  $ and $A=\psi\left(  t\right)  B$, where $D^{-3}$ is defined by \eqref{k10}. If
$L=i\partial_{t}+\partial_{x}^{2}-\partial_{x}^{4}$, we
write%
\begin{align*}
&  \left(  Lu_{n},A^{\ast}u_{n}\right)  _{L^{2}\left(  \mathbb{T\times}\left(
0,T\right)  \right)  }-\left(  Au_{n},Lu_{n}\right)  _{L^{2}\left(
\mathbb{T\times}\left(  0,T\right)  \right)  }\\
&  =\left(  \left[  A,\partial_{x}^{2}-\partial_{x}^{4}\right]  u_{n},u_{n}\right)  _{L^{2}\left(  \mathbb{T\times}\left(
0,T\right)  \right)  }-i\left(  \psi^{\prime}\left(  t\right)  Bu_{n}%
,u_{n}\right)  _{L^{2}\left(  \mathbb{T\times}\left(  0,T\right)  \right)  }%
\end{align*}
and deduce that%
\begin{align*}
\left\vert \left(  Au_{n},Lu_{n}\right)  _{L^{2}\left(  \mathbb{T\times
}\left(  0,T\right)  \right)  }\right\vert  &  =\left\vert \left(
Au_{n},f_{n}\right)  _{L^{2}\left(  \mathbb{T\times}\left(  0,T\right)
\right)  }\right\vert \\
&  \leq\left\Vert Au_{n}\right\Vert _{X_{b,-r}^{T}}\left\Vert f_{n}\right\Vert
_{X_{-b,r}^{T}}\\
&  \leq C\left\Vert u_{n}\right\Vert _{X_{b,r+2\rho-3+3b}^{T}}\left\Vert
f_{n}\right\Vert _{X_{-b,r}^{T}}\text{,}%
\end{align*}
since $r+2\rho-3+3b\leq r$. The same estimates for the other terms imply that%
\[
\left\vert \left(  \left[  A,\partial_{x}^{2}-\partial_{x}^{4}\right]  u_{n},u_{n}\right)  _{L^{2}\left(  \mathbb{T\times
}\left(  0,T\right)  \right)  }\right\vert \leq C\text{.}%
\]
Note that%
\begin{equation}\label{k27}
\begin{split}
\left[  A,\partial_{x}^{2}-\partial_{x}^{4}\right] 
  =&4\psi\left(  t\right) D^{2s-3}  \left(  \partial_{x}^{3}\varphi\right)
\partial_{x}+12\psi\left(  t\right) D^{2s-3} \left(  \partial_{x}^{2}%
\varphi\right)  \partial_{x}^{2}+4\psi\left(  t\right) D^{2s-3} \left(  \partial_{x}\varphi\right)  \partial_{x}^{3}\\
&  -2\psi\left(  t\right) D^{2s-3} \left(  \partial_{x}\varphi\right)
\partial_{x} -\psi\left(  t\right) D^{2s-3} \left(  \partial_{x}^{2}\varphi-\partial_{x}^{4}\varphi\right) \\
=&:\sum\limits_{i=1}^{5}I_{i}.
\end{split}
\end{equation}
Also, observe that
\begin{equation} \label{k29}
2s-3+2=2r+2\rho-1\leq2r
\end{equation}
and%
\begin{equation}\label{k30}
2s-3+1=2r+2\rho-2\leq2r\text{.} 
\end{equation}
Now, we will bound the terms of \eqref{k27}. As \eqref{k29} and \eqref{k30} are verified, taking $$\rho\leq\min\left\{  \frac{3}{2}\left(  1-b\right),\frac{1}{2}\right\},$$  we have that 
\begin{equation}
\begin{split}
\left\vert \left(  I_{1}u_{n},u_{n}\right)  _{L^{2}\left(  \mathbb{T\times
}\left(  0,T\right)  \right)  }\right\vert  &  \leq C\left\Vert \psi
D^{2s-3}\left(  \partial_{x}^{3}\varphi\right)  \partial_{x}u_{n}\right\Vert _{L^{2}\left(  0,T;H^{-r}\left(  \mathbb{T}\right)  \right)
}\left\Vert u_{n}\right\Vert _{L^{2}\left(  0,T;H^{r}\left(  \mathbb{T}%
\right)  \right)  }\\
&  \leq C\left\Vert u_{n}\right\Vert _{L^{2}\left(  0,T;H^{r}\left(
\mathbb{T}\right)  \right)  }^{2}\leq C\text{,}%
\\
\left\vert \left(  I_{2}u_{n},u_{n}\right)  _{L^{2}\left(  \mathbb{T\times
}\left(  0,T\right)  \right)  }\right\vert  &  \leq C\left\Vert \psi
D^{2s-3}\left(  \partial_{x}\varphi\right)  \partial_{x}^{2}%
u_{n}\right\Vert _{L^{2}\left(  0,T;H^{-r}\left(  \mathbb{T}\right)  \right)
}\left\Vert u_{n}\right\Vert _{L^{2}\left(  0,T;H^{r}\left(  \mathbb{T}%
\right)  \right)  }\\
&  \leq C\left\Vert u_{n}\right\Vert _{L^{2}\left(  0,T;H^{r}\left(
\mathbb{T}\right)  \right)  }^{2}\leq C,
\end{split}
\end{equation}
and
\begin{align*}
\left\vert \left(  I_{5}u_{n},u_{n}\right)  _{L^{2}\left(  \mathbb{T\times
}\left(  0,T\right)  \right)  }\right\vert  &  =\left\vert \left(  \psi\left(
t\right)  D^{2s-3}\left(\partial_{x}^{2}\varphi-\partial_{x}^{4}\varphi\right)  u_{n},u_{n}\right)  _{L^{2}\left(
\mathbb{T\times}\left(  0,T\right)  \right)  }\right\vert \\
&  \leq\left\Vert \psi\left(  t\right)  D^{2s-3}\left( \partial_{x}^{2}\varphi-\partial_{x}^{4}\varphi\right)u_{n}\right\Vert _{L^{2}\left(  0,T;H^{-r}\left(  \mathbb{T}\right)  \right)
}\left\Vert u_{n}\right\Vert _{L^{2}\left(  0,T;H^{r}\left(  \mathbb{T}%
\right)  \right)  }\\
&  \leq C\left\Vert u_{n}\right\Vert _{L^{2}\left(  0,T;H^{r}\left(
\mathbb{T}\right)  \right)  }^{2}\leq C\text{,}%
\end{align*}
for any $n\geq1$. Similarly of $I_1$ estimate we can get
\begin{equation*}
|I_4|\leq C.
\end{equation*}


Finally, we will  control $I_3$.
For any $\chi\in C_{0}^{\infty}\left(  \omega\right)  $, we get that
\begin{align}
\left(  \psi\left(  t\right)  D^{2s-3}\chi^{2}\partial_{x}^{3}u_{n}%
,u_{n}\right)  _{L^{2}\left(  \mathbb{T\times}\left(  0,T\right)  \right)  }
  =&\left(  \psi\left(  t\right)  D^{s-3}\chi\partial_{x}^{3}u_{n},\chi
D^{s}u_{n}\right)  +\left(  \psi\left(  t\right)  \left[  D^{s-3},\chi\right]
\chi\partial_{x}^{3}u_{n},D^{s}u_{n}\right) \nonumber\\
  =&\left(  \psi\left(  t\right)  D^{s-3}\chi\partial_{x}^{3}u_{n},D^{s}\chi
u_{n}\right)  +\left(  \psi\left(  t\right)  D^{s-3}\chi\partial_{x}^{3}u_{n},\left[  \chi,D^{s}\right]  u_{n}\right) \label{k31}\\
&  +\left(  \psi\left(  t\right)  \left[  D^{s-3},\chi\right]  \chi\partial
_{x}^{3}u_{n},D^{s}u_{n}\right):=\tilde{I}_{1}+\tilde{I}_{2}+\tilde{I}_{3}\text{.}\nonumber
\end{align}
In this moment, we need control the right hand side of \eqref{k31}. First, note that we infer from the assumptions that%
\[
\chi u\in L_{loc}^{2}\left(0,T ;H^{s}\left(  \mathbb{T}%
\right)  \right)
\]
and
\[
\chi\partial_{x}^{3}u\in L_{loc}^{2}\left( 0,T ;H^{s-3}\left(
\mathbb{T}\right)  \right)  \text{.}%
\]
Then, as $s=r+\rho\leq r+1$, we have
\[
\chi u_{n}=\Theta_{n}\chi u+\left[  \chi,\Theta_{n}\right]  u\in L_{loc}%
^{2}\left( 0,T ;H^{s}\left(  \mathbb{T}\right)  \right)
\text{,}%
\]
due \cite[Lemma A.3]{Laurent-esaim}. Applying the same argument to
$\chi\partial_{x}^{3}u_{n}$, follows that
\begin{equation}
\left\vert \tilde{I}_{1}\right\vert \leq C\text{.} \label{k32}%
\end{equation}
Moreover, from  \cite[Lemma A.1]{Laurent-esaim} and the fact that $u\in
L^{2}\left(  0,T;H^{r}\left(  \mathbb{T}\right)  \right)  $, the second one can be bounded in the following way
\begin{equation}\label{k33}
\begin{split}
\left\vert \tilde{I}_{2}\right\vert  &  =\left\vert \left(  \psi\left(
t\right)  D^{s-3}\chi\partial_{x}^{3}u_{n},\left[  \chi,D^{s}\right]
u_{n}\right)  _{L^{2}\left(  \mathbb{T\times}\left(  0,T\right)  \right)
}\right\vert\\
&  =\left\vert \left(  \psi\left(  t\right)  D^{r+\rho-3}\chi\partial_{x}%
^{3}u_{n},\left[  \chi,D^{s}\right]  u_{n}\right)  _{L^{2}\left(
\mathbb{T\times}\left(  0,T\right)  \right)  }\right\vert \\
&  =\left\vert \left(  \psi\left(  t\right)  D^{\rho}D^{r-3}\chi\partial
_{x}^{3}u_{n},\left[  \chi,D^{s}\right]  u_{n}\right)  _{L^{2}\left(
\mathbb{T\times}\left(  0,T\right)  \right)  }\right\vert \\
&  \leq\left\Vert \psi\left(  t\right)  D^{r-3}\chi\partial_{x}^{3}%
u_{n}\right\Vert _{_{L^{2}\left(  \mathbb{T\times}\left(  0,T\right)  \right)
}}\left\Vert D^{\rho}\left[  \chi,D^{s}\right]  u_{n}\right\Vert _{L^{2}\left(
\mathbb{T\times}\left(  0,T\right)  \right)  }\\
&  \leq C\left\Vert u_{n}\right\Vert _{L^{2}\left(  0,T;H^{r}\left(
\mathbb{T}\right)  \right)  }\left\Vert u_{n}\right\Vert _{L^{2}\left(
0,T;H^{\rho+s-1}\left(  \mathbb{T}\right)  \right)  }\leq C\text{.}
\end{split}
\end{equation}
Lastly, by similar computations, we ensure that%
\begin{equation}\label{k344}
\left\vert \tilde{I}_{3}\right\vert \leq C\text{.}%
\end{equation}
Consequently,
\[
\left\vert \left(  \psi\left(  t\right)  D^{2s-3}\chi^{2}\partial_{x}^{3}%
u_{n},u_{n}\right)  _{L^{2}\left(  \mathbb{T\times}\left(  0,T\right)
\right)  }\right\vert \leq C\text{,}%
\]
for any $n\geq1$. Then, writing $\partial_{x}\varphi=\chi^{2}\left(  x\right)  -\chi^{2}\left(  x-x_{0}\right)$,  from (\ref{k31}), \eqref{k32}, \eqref{k33} and \eqref{k344} yields,
\[
\left\vert \left(  \psi\left(  t\right)  D^{2s-3}\chi^{2}\left(  \cdot
-x_{0}\right)  \partial_{x}^{3}u_{n},u_{n}\right)  _{L^{2}\left(
\mathbb{T\times}\left(  0,T\right)  \right)  }\right\vert \leq C\text{,}%
\]
for all $n\geq1$.
To conclude the proof, is necessary to use a partition of unity as in the
proof of Proposition \ref{prop_b}, to obtain%
\[
\left\vert \left(  \psi\left(  t\right)  D^{2s-3}\partial_{x}^{3}u,u\right)
_{L^{2}\left(  \mathbb{T\times}\left(  0,T\right)  \right)  }\right\vert \leq
C\text{,}%
\]
that is%
\[
\int_{0}^{T}\psi\left(  t\right)  \left(  \sum\limits_{k\neq0}\left\vert
k\right\vert ^{2s}\left\vert \hat{u}\left(  k,t\right)  \right\vert
^{2}dt\right)  \leq C\text{,}%
\]
or equivalently,
\[
\left\Vert u\right\Vert _{L_{loc}^{2}\left([0,T];H^{s}\left(
\mathbb{T}\right)  \right)  }^{2}\leq C\text{.}%
\]
Thus, the proof is complete.
\end{proof}

\section{Unique continuation property\label{sec6}}

We present, in this section, the unique continuation result for 4NLS. However, before to enunciate the UCP, let us prove a auxiliary lemma which is a consequence of Proposition \ref{prop_b}.

\begin{lemma}
\label{unique_a}Let $u\in X_{b,0}^{T}$ be a solution of%
\begin{equation}\label{k34}
i\partial_{t}u+\partial_{x}^{2}u-\partial_{x}^{4}u+\lambda \left\vert
u\right\vert^2u=0 \ \ \text{on \ }\mathbb{T\times}\left(  0,T\right)
\text{.} 
\end{equation}
Here $b>\frac{1}{2}$ and we assume that $u\in C^{\infty}\left(  \omega\times\left(  0,T\right)  \right)
$, where $\omega\subset\mathbb{T}$ nonempty set. Then,%
\[
u\in C^{\infty}\left(  \mathbb{T\times}\left(  0,T\right)  \right)  \text{.}%
\]
\end{lemma}
\begin{proof}
Note that $\lambda \left\vert u\right\vert^2u\in X_{-b,0}^{T}$, by Lemma
\ref{trilinear_estimates}. Thus, from Proposition \ref{prop_b}, we get%
\[
u\in L_{loc}^{2}([0,T];H^{1+\frac{3}{2}(1-b)}(
\mathbb{T}))  \text{.}%
\]
Choose $t_{0}$ such that $u\left(  t_{0}\right)  \in H^{1+\frac{3}{2}(1-b)}\left(
\mathbb{T}\right)  $. We can then solve \eqref{k34} in $X_{b%
,1+\frac{3}{2}(1-b)}^{T}$ with the initial data $u\left(  t_{0}\right)  $. By
uniqueness of solution in $X_{b,0}^{T}$, we conclude that $u\in
X_{b,1+\frac{3}{2}(1-b)}^{T}$. An iterated application of Proposition
\ref{prop_b} give as%
\[
u\in L^{2}\left(  0,T;H^{r}\left(  \mathbb{T}\right)  \right)  \text{,
}\quad \forall r\in\mathbb{R}\text{,}%
\]
and, hence $u\in C^{\infty}\left(  \mathbb{T\times}\left(  0,T\right)
\right)  $.
\end{proof}

The UCP is presented as follows:
\begin{proposition}
[\textit{Unique continuation property}]\label{UCP}For every $T>0$ and $\omega$ any nonempty open set of $\mathbb{T}$, the only solution $u\in C^{\infty}([0,T]\times\mathbb{T})$ of the system%
\[
\left\{
\begin{array}
[c]{lll}%
i\partial_{t}u+\partial_{x}^{2}u-\partial_{x}^{4}u=b(x,t)u&  & \text{on }\mathbb{T\times}\left(  0,T\right)
\text{,}\\
u=0 &  & \text{on }\omega\times\left(  0,T\right)  \text{,}%
\end{array}
\right.
\]
where $b(x,t)\in C^{\infty}([0,T]\times\mathbb{T})$, is the trivial one 
\[
u\left(  x,t\right)  =0\text{ \ on \ }\mathbb{T\times}\left(  0,T\right)
\text{.}%
\]
\end{proposition}

\begin{proof} Proposition \ref{UCP} is a direct consequence of the Carleman estimate for the operator $P=i\partial_{t}+\partial_{x}^{2}-\partial_{x}^{4}$, proved by Zheng \cite[Theorem 1.1.]{zheng} (see also \cite[Corollary 6.1]{Isakov}), together with Lemma \ref{unique_a}. 
\end{proof}

\begin{corollary}\label{UCP_main}
Let $\omega$ be any nonempty open set of  $\mathbb{T}$ and $u\in X^T_{\frac{1}{2},0}$ solution of
\[
\left\{
\begin{array}
[c]{lll}%
i\partial_{t}u+\partial_{x}^{2}u-\partial_{x}^{4}u=\lambda \left\vert
u\right\vert^2u &  & \text{on }\mathbb{T\times}\left(  0,T\right)
\text{,}\\
u=0 &  & \text{on }\omega\times\left(  0,T\right)  \text{,}%
\end{array}
\right.
\]
then $u\left(  x,t\right)  =0\text{ \ on \ }\mathbb{T\times}\left(  0,T\right)$.
\end{corollary}
\begin{proof}
By using Lemma \ref{unique_a}, we infer that $u\in C^{\infty}\left(  \mathbb{T\times}\left(  0,T\right)  \right)  \text{.}$ An application of Proposition \ref{UCP} give us $u=0$, as desired.
\end{proof}

\begin{remark}\label{rmk_UCP}
Proposition \ref{UCP} assures us that for $u\in X_{b,0}^{T}$ solution of%
\[
\left\{
\begin{array}
[c]{lll}%
i\partial_{t}u+\partial_{x}^{2}u-\partial_{x}^{4}u=0 &  & \text{on }\mathbb{T\times}\left(  0,T\right)
\text{,}\\
u=0 &  & \text{on }\omega\times\left(  0,T\right)  \text{,}%
\end{array}
\right.
\]
we also have $u\left(  x,t\right)  =0\text{ \ on \ }\mathbb{T\times}\left(  0,T\right)$.
\end{remark}

\section{Stabilization: Global result\label{sec7}}

This section is to establish the main result of this article. The propagation and unique continuation property will play a keys role for this study. We are concerned with stability properties of the following system
\begin{equation}
\left\{
\begin{array}
[c]{lll}%
i\partial_{t}u+\partial_{x}^{2}u-\partial_{x}^{4}u+ia^2u=\lambda \left\vert
u\right\vert^2u &  & \text{on }\mathbb{T\times}\left(
0,T\right)  \text{,}\\
u\left(  x,0\right)  =u_{0}\left(  x\right)  &  & \text{on }\mathbb{T}\text{,}%
\end{array}
\right.  \label{k35}%
\end{equation}
where $\lambda\in\mathbb{R}$ and $u_{0}\in L^{2}\left(
\mathbb{T}\right)  $, in $L^2$--level. 
\subsection{Proof of Theorem \ref{main}} Theorem \ref{main} is a consequence of the following \textit{observability inequality}:

\vspace{0.2cm}

\noindent\textit{Let $T>0$ and $R_{0}>0$ be given. There exists a constant
$\gamma>0$ such that for any $u_{0}\in L^{2}\left(  \mathbb{T}\right)  $
satisfying $\left\Vert u_{0}\right\Vert _{L^2(\mathbb{T})}\leq R_{0}$,  the corresponding solution $u$ of (\ref{k35}) satisfies%
\begin{equation}
\left\Vert u_{0}\right\Vert _{L^2(\mathbb{T})}^{2}\leq\gamma\int_{0}^{T}\left\Vert
au\right\Vert _{L^2(\mathbb{T})}^{2} dt\text{.} \label{k46}%
\end{equation}}

\vspace{0.2cm}

In fact, if \eqref{k46} holds, then follows from the energy estimate
\[
\left\Vert u\left(\cdot,  t\right)  \right\Vert _{L^2(\mathbb{T})}^{2}=\left\Vert u\left(\cdot,0\right)  \right\Vert _{L^2(\mathbb{T})}^{2}-\int_{0}^t\left\Vert au
\right\Vert _{L^2(\mathbb{T})}^{2}(\tau)d\tau, \quad\forall t\geq0
\]
that $$\left\Vert u\left(\cdot,  T\right)  \right\Vert _{L^2(\mathbb{T})}^{2}\leq(1-\gamma^{-1})\left\Vert u_0  \right\Vert _{L^2(\mathbb{T})}^{2}.$$ Thus, $$\left\Vert u\left(\cdot,  mT\right)  \right\Vert _{L^2(\mathbb{T})}^{2}\leq(1-\gamma^{-1})^m\left\Vert u_0  \right\Vert _{L^2(\mathbb{T})}^{2},$$
which gives
\begin{equation*}
\left\Vert u\left(  \cdot,t\right)  \right\Vert _{L^2(\mathbb{T})}\leq Ce^{-\gamma
t}\left\Vert u_{0}\right\Vert _{L^2(\mathbb{T})}\text{, }\forall t>0\text{.} 
\end{equation*}

Finally, we obtain a constant $\gamma$ independent
of $R_0$ by noticing that for $t>c\left(\left\Vert u_0  \right\Vert _{L^2(\mathbb{T})}\right)$, the $L^2$ norm of $u(\cdot,t)$ is smaller than $1$, so that we can take the $\gamma$ corresponding to $R_0 = 1$, thus proving the result. \qed

\subsection{Proof of the observability inequality}  If  \eqref{k46} does not occurs, there exist a sequence $\left\{u_{n}\right\}  _{n\in\mathbb{N}}=u_{n}$ solution of (\ref{k35}) satisfying%
\[
\left\Vert u_{n}\left(  0\right)  \right\Vert _{L^2(\mathbb{T})}\leq R_{0}%
\]
and%
\begin{equation} \label{k47}%
\int_{0}^{T}\left\Vert au_{n}\right\Vert _{L^2(\mathbb{T})}^{2}dt<\frac{1}{n}\left\Vert
u_{0,n}\right\Vert _{L^2(\mathbb{T})}^{2}\text{,}
\end{equation}
where $u_{0,n}=u_{n}\left(  0\right)  $. Since $\gamma_{n}:=\left\Vert u_{0,n}\right\Vert
_{L^2(\mathbb{T})}\leq R_{0}$, one can choose a subsequence of $\gamma_{n}=\left\{  \gamma
_{n}\right\}  _{n\in\mathbb{N}}$, still denote by $\gamma_{n}$, such that,
\[
\lim_{n\rightarrow\infty}\gamma_{n}=\gamma\text{.}%
\]
Thus, we will analyze two cases for $\gamma$: $\gamma>0$ or $\gamma=0$.  In both cases we will get a contradiction. 

\vspace{0.2cm}
\noindent \textit{\textbf{Case one:}} $\lim_{n\rightarrow\infty}\gamma_{n}=\gamma>0$:
\vspace{0.2cm}

Observe that  $u_{n}$ is bounded in $L^{\infty}\left(  0,T;L^{2}%
\left(  \mathbb{T}\right)  \right)$ and, therefore, in $X^T_{b,0}$, for $b>\frac12$. Then, as $X^T_{b,0}$ is a separable Hilbert space we can extract a subsequence such that $$u_n \rightharpoonup u \quad\text{in} \quad X^T_{b,0}, $$ for some $u\in X^T_{b,0}$. By compact embedding, as we have $b<1$ and $-b<0$, we can (also) extract a subsequence such that we have strong convergence in $X^T_{-b,-1+b}$. Now, we prove that the weak limit $u$ is a solution of \eqref{k35}.  
 Thus, $\left\vert u_n\right\vert^2u_n $ is
bounded in $X^T_{-b',0}$, for $b'>\frac{5}{16}$.

Note that, there is a subsequence $u_{n}$, still denote by $u_{n}$, such that%
\[
\left\vert u_n\right\vert^2u_n  \rightharpoonup f\quad\text{in} \quad X^T_{-b',0},\quad \text{for}\quad b'>\frac{5}{16}
\]
and
\[
\left\vert u_n\right\vert^2u_n  \rightarrow f \quad\text{in} \quad X^T_{-1+b,-b},\quad \text{for}\quad \frac{5}{16}<b<1.%
\]
Moreover, from \eqref{k47} it follows that%
\begin{equation*}
\int_{0}^{T}\left\Vert au_{n}\right\Vert _{L^2(\mathbb{T})}^{2}dt\longrightarrow\int_{0}%
^{T}\left\Vert au\right\Vert _{L^2(\mathbb{T})}^{2}dt=0\text{,} 
\end{equation*}
which implies that $u\left(  x,t\right)  =0$ on
$\omega\times\left(  0,T\right)  $.
Therefore, letting $n\rightarrow\infty$, we obtain from (\ref{k35}) that%
\begin{equation*}
\left\{
\begin{array}
[c]{lll}%
i\partial_{t}u+\partial_{x}^{2}u-\partial_{x}^{4}u=f &  &
\text{on }\mathbb{T\times}\left(  0,T\right)  \text{,}\\
u\left(  x,t\right)  =0&  & \text{on }\mathbb{\omega\times
}\left(  0,T\right)  \text{.}%
\end{array}
\right.  
\end{equation*}
We affirm that $$f=-ia^2u+\lambda \left\vert u\right\vert^2u.$$ In fact,
let $w_{n}=u_{n}-u$ and $f_{n}=-ia^2u_n+\lambda \left\vert u_n\right\vert^2u_n  -f$. Remark that from \eqref{k47},
\begin{equation*}
\int_{0}^{T}\left\Vert aw_{n}\right\Vert _{L^2(\mathbb{T})}^{2}dt\longrightarrow0\text{.}
\end{equation*}
Thus, 
$$f_n \rightarrow 0 \quad\text{in} \quad X^T_{-1+b,-b}. $$
It also implies 
\[
u_{n}\rightarrow0\text{ in }L^{2}\left(  0,T;L^{2}\left(  \omega
\right)  \right) 
\]
and 
\[
w_{n}\rightarrow0\text{ in }L^{2}\left(  0,T;L^{2}\left(  \omega
\right)  \right) .
\]
Applying Proposition \ref{prop_a}, we get%
\[
w_{n}\longrightarrow0\text{ in }L_{loc}^{2}\left([0,T]
;L^{2}\left(  \mathbb{T}\right)  \right)  \text{.}%
\]
Then, we can pick one $t_0\in [0,T]$ such that $w_n(t_0)$ tends to $0$ strongly in $L^2(\mathbb{T})$.

Let $v$ the solution of
\begin{equation}\label{eqv}
\left\{
\begin{array}
[c]{lll}%
i\partial_{t}v+\partial_{x}^{2}v-\partial_{x}^{4}v+ia^2v=\lambda \left\vert
v\right\vert^2v &  & \text{on }\mathbb{T\times
}\left(  0,T\right)  \text{,}\\
v\left(  t_0\right)  =u(t_0). &  & %
\end{array}
\right.
\end{equation}
We claim  that $u=v$. Indeed, by Theorem \ref{LWP} we have that the map data-to-solution of \eqref{k35} is locally Lipschitz continuous. Since $u_n(t_0)\rightarrow v(t_0)$ in $L^2(\mathbb{T})$ and $ia^2u_n\rightarrow ia^2u$ in $L^2([0,T];L^2(\mathbb{T}))$, we get $u_n\rightarrow v$ in $X_{0,b}^T$, thus $u=v$ and $u$ is a solution of \eqref{eqv}. Unique continuation property, Corollary \ref{UCP_main}, implies $u=0$. It follows that $\|u_n(0)\|_{L^2(\mathbb{T})}\rightarrow 0$, which leads a contradiction of our hypothesis $\alpha>0$.

\vspace{0.2cm}
\noindent \textit{\textbf{Case two:}} $\lim_{n\rightarrow\infty}\gamma_{n}=\gamma=0$:
\vspace{0.2cm}

Consider the following change of variable $v_{n}=\frac{u_{n}}%
{\gamma_{n}}$, $\forall n\geq1$. Thus, $v_n$ satisfies 
\[
i\partial_{t}v_n+\partial_{x}^{2}v_n-\partial_{x}^{4}v_n+ia^2v_n=\lambda\gamma_n^2 \left\vert
v_n\right\vert^2v_n \text{,}%
\]
and
\begin{equation}\label{k50}
\int_{0}^{T}\left\Vert av_{n}\right\Vert _{L^2(\mathbb{T})}^{2}dt<\frac{1}{n}.
\end{equation}
Then, we have
\begin{equation}\label{k51}
\left\Vert v_{n}\left(  0\right)  \right\Vert _{L^2(\mathbb{T})}=1\text{.} 
\end{equation}
Observe that $v_n:=\left\{  v_{n}\right\}  _{n\in\mathbb{N}}$ is bounded in
$L^{\infty}\left(  0,T;L^{2}\left(  \mathbb{T}\right)  \right)  \cap
X_{b,0}^{T}$. Thus, we can extract a subsequence, still denote by
$v_{n}$, such that%
\[
v_{n}\rightharpoonup v\text{ in }X_{b,0}^{T}.
\]
Furthermore, by Duhamel formula and multilinear estimates \eqref{trilinear_estimates}, we obtain  $$\left\Vert v_n \right\Vert_{X^T_{b,0}}\leq C\left\Vert v_n(0)\right\Vert_{L^2(\mathbb{T})}+CT^{1-b-b'}\left(\left\Vert v_n \right\Vert_{X^T_{b,0}}+\gamma_n^2\left\Vert v_n \right\Vert_{X^T_{b,0}}^3\right),$$
for $0<b'\leq\frac12\leq b$ and $b+b'\leq 1$.
 
 If we take $CT^{1-b-b'}<1/2$, independent of $v_n$, we get $$\left\Vert v_n \right\Vert_{X^T_{b,0}}\leq C+C\gamma_n^2\left\Vert v_n \right\Vert_{X^T_{b,0}}^3.$$
Lemma \ref{lemma2} states that $\left\Vert v_n \right\Vert_{X^T_{b,0}}$ is continuous in $T$. Since it is bounded near $t=0$ and $\gamma_n\to 0$, we obtain by a classical boot strap argument (see, e.g, \cite[Lemma 2.2]{BaGe1999}) that $v_n$ is bounded on $X^T_{b,0}$. Using Lemma \ref{lemma4}, we can conclude that it is bounded in $X^T_{b,0}$ even for large $T$. Thus, $$\gamma_n^2 \left\vert v_n\right\vert^2v_n \to 0 \quad\text{in} \quad X^T_{-b',0}$$ and so $$\gamma_n^2 \left\vert v_n\right\vert^2v_n \to 0 \quad\text{in} \quad X^T_{-b,-1+b}.$$ Then, we can extract a subsequence such that  $$v_{n}  \rightharpoonup v \quad\text{in} \quad X^T_{b,0}$$ and $$v_{n}  \to v \quad\text{in} \quad X^T_{-1+b,-b}.$$
Therefore, the weak limit $v$ satisfies
\begin{equation*}
\left\{
\begin{array}
[c]{lll}%
i\partial_{t}v+\partial_{x}^{2}v-\partial_{x}^{4}v+ia^2v=0 &  &
\text{on }\mathbb{T\times}\left(  0,T\right)  \text{,}\\
v\left(  x,t\right)  =0&  & \text{on }\mathbb{\omega\times
}\left(  0,T\right)  \text{,}%
\end{array}
\right. 
\end{equation*}
which implies that $v\left(  x,t\right)$ is the trivial solution, that is, $v\left(  x,t\right)=0$, thanks to remark \ref{rmk_UCP} of Proposition \ref{UCP}.

Argument of contradiction \eqref{k50} yields that
\begin{equation}\label{ax}
ia^2v_{n}\longrightarrow0 \quad\text{in} \quad L^2(0,T;L^2(\mathbb{T})),
\end{equation}
and so 
$$ia^2v_{n}\longrightarrow0 \quad\text{in} \quad X^T_{-1+b,-b}.$$
An application of Proposition \ref{prop_a}, as in the case one $\gamma>0$, ensures that
\begin{equation}\label{ax_1}
v_{n}\longrightarrow0\text{ in }L^{2}_{loc}\left( [ 0,T];
L^{2}\left(  \mathbb{T}\right)  \right)  \text{.}
\end{equation}
From the energy estimate for $t_{0}\in\left(  0,T\right)  $, we get
\[
\left\Vert v_{n}\left(  0\right)  \right\Vert _{L^2(\mathbb{T})}^{2}=\left\Vert v_{n}\left(
t_{0}\right)  \right\Vert _{L^2(\mathbb{T})}^{2}+\int_{0}^{t_{0}}\left\Vert av_{n}%
\right\Vert _{L^2(\mathbb{T})}^{2}dt\text{.}%
\]
Passing the limit on the last equality, by using \eqref{ax} and \eqref{ax_1}, we have that $\left\Vert v_{n}\left(  0\right)  \right\Vert
_{L^2(\mathbb{T})}\rightarrow0$,  which contradicts \eqref{k51}. Therefore, the proof is archived. \qed

\subsection*{Acknowledgments} R. A. Capistrano-Filho was partially supported by CNPq (Brazil) by grant number 306475/2017-0 and "Propesq (UFPE) - Edital Qualis A".   This work was carried out during visits of the first author to the Federal University of Alagoas and visits of the second author to the Federal University of Pernambuco. The authors would like to thank both Universities for its hospitality.


\begin{thebibliography}{99}                                                                                               
\bibitem{AkReSa} Aksas, B., Rebiai, S-E. (2017). Uniform stabilization of the fourth order Schr\"odinger equation. \textit{J. Math. Anal. Appl.} 2:1794--1813.

\bibitem{BaGe1999} Bahouri, H., Gérard, P. (1999). High frequency approximation of critical nonlinear wave equations. \textit{Am. J. Math. }121:131--175.

\bibitem{Ben} Ben-Artzi, M., Koch, H., Saut, J.-C. (2000). Dispersion estimates for fourth order Schr\"odinger equations. \textit{C. R. Acad. Sci. Paris S\'er. I Math} 330:87--92.

\bibitem{Berg} Bergh, J.,  L\"{o}fstr\"{o}m, J. (1976). \textit{ Interpolation spaces. An introduction}. Grundlehren der Mathematishen Wissenschaften, No. 223. Berlin: Springer-Verlag.

\bibitem{Bourgain} Bourgain, J. (1993). Fourier transform restriction phenomena for certain lattice subsets and applications to nonlinear evolution equations. Part I: The Schr\"odinger equations. \textit{Geometric and Funct. Anal.} 3:107--156.

\bibitem{Burq}Burq, N., Gérard, P., Tzvetkov, N. (2002). An instability property of the nonlinear Schr\"odinger equation on $S^d$. \textit{Math. Res. Lett.} 9:323--335.

\bibitem{Burq1}Burq, N., Thomann, L., Tzvetkov, N. (2013). Long time dynamics for the one dimensional non linear Schr\"odinger equation. \textit{Ann. Inst. Fourier (Grenoble)} 63:2137--2198.

\bibitem{CuiGuo} Cui, S., Guo, C. (2007). Well-posedness of higher-order nonlinear Schr\"odinger equations in Sobolev spaces $H^s(R^n)$ and applications \textit{Nonlinear Analysis} 67:687--707.

\bibitem {dehman-lebeau-zuazua}Dehman, B., Lebeau, G., Zuazua, E. (2003). Stabilization and control for the subcritical semilinear wave equation. \textit{Ann. Sci. de l’\'Ecole Normale Sup\'erieure} 36:525--551.

\bibitem {dehman-gerard-lebeau}Dehman B., G\'{e}rard, P., Lebeau, G. (2006).
Stabilization and control for the nonlinear Schr\"{o}dinger equation on a
compact surface. \textit{Mathematische Zeitschrift} 254:729--749.

\bibitem {dehman-lebeau} Dehman, B., Lebeau, G. (2009). Analysis of the HUM Control Operator and Exact Controllability for Semilinear Waves in Uniform Time. \textit{SIAM J. Control Optim.} 48:521--550.

\bibitem{DolRus1977} Dolecki, S., Russell, D.L. (1977). A general theory
of observation and control. \textit{SIAM J. Control Opt}. 15:185--220.

\bibitem{FiIlPa}Fibich, G., Ilan, B., Papanicolaou, G. (2002). Self-focusing with fourth-order dispersion. \textit{SIAM J. Appl. Math.} 62:1437--1462.

\bibitem{Ginibre}Genibre, J. (1994). Le probl\'eme de Cauchy pour des EDP semi-lin\'eaires p\'eriodiques en variables d'espace. \textit{in S\'eminaire Bourbaki } 37,
expos\'e 796:163--187.

\bibitem{Peng} Gao, P. Carleman estimates for forward and backward stochastic fourth order Schr\"odinger equations and their applications. arXiv:1703.03629 [math.OC].

\bibitem{Isakov}Isakov, V. (1993). Carleman type estimates in an anisotropic case and applications. \textit{J. Differ. Equ.} 105:217--238.

\bibitem{Karpman} Karpman, V.I. (1996). Stabilization of soliton instabilities by higher-order dispersion: fourth order nonlinear Schr\"odinger-type equations. \textit{Phys. Rev. E} 53:1336--1339.

\bibitem{KarSha} Karpman, V.I., Shagalov, A.G. (2000). Stability of soliton described by nonlinear Schrdinger type equations with higher-order dispersion. \textit{Physica D} 144:194--210.

\bibitem{Laurent-esaim}Laurent, C. (2010). Global controllability and stabilization
for the nonlinear Schr\"{o}dinger equation on an interval. \textit{ESAIM Control
Optim. Calc. Var.} 16:356--379.

\bibitem{LaurentLinaresRosier} Laurent, C., Linares, F., Rosier, R. (2015). Control and stabilization of the Benjamin-Ono equation  in $L^2(T)$. \textit{Arch. Mech. Anal.} 218:1531--1575.

\bibitem {Laurent}Laurent, C., Rosier, L., Zhang, B.-Y. (2010). Control and
Stabilization of the Korteweg-de Vries Equation on a Periodic Domain. \textit{Commun.
in Partial Differential Equations} 35:707--744.


\bibitem{lions1} Lions, J.-L. (1988). Exact controllability, stabilization and perturbations for distributed systems. \textit{SIAM Rev.} 30:1--68.

\bibitem{NataliPastor} Natali, F., Pastor, A. (2015). The fourth-order dispersive nonlinear Schr\"odinger equation: orbital stability of a standing wave.\textit{ SIAM J. Appl. Dyn. Syst.} 14:1326--1347. 

\bibitem{Paus} Pausader,B. (2007). Global well-posedness for energy critical fourth-order Schr\"odinger equations in the radial case. \textit{Dynamics of PDE 4} 197--225.

\bibitem{Paus1} Pausader, B. (2009). The cubic fourth-order Schr\"odinger equation, \textit{J. Funct. Anal.} 256:2473--2517.

\bibitem {Tao}Tao, T. (2006). \textit{Nonlinear Dispersive Equations, Local and Global
Analysis}. CBMS Regional Conference Series in Mathematics. Providence, RI:
American Mathematical Society.

\bibitem{tsutsumi} Tsutsumi, T. (2014). Strichartz estimates for Schr\"odinger equation of fourth order with periodic boundary condition. \textit{Kyoto university} 11pp.

\bibitem{Tzvetkov}Tadahiro, O., Tzvetkov, N. (2016). Quasi-invariant Gaussian measures for the cubic fourth order nonlinear Schr\"odinger equation. \textit{Probab. Theory Relat. Fields} 169:1121--1168.

\bibitem{WenChaiGuo1} Wen, R., Chai, S., Guo, B.-Z, (2014). Well-posedness and exact controllability of fourth order Schr\"odinger equation with boundary control and collocated observation. \textit{SIAM J. Control Optim.} 52:365--396.

\bibitem{WenChaiGuo}Wen, R., Chai, S., Guo, B.-Z. (2016). Well-posedness and exact controllability of fourth-order Schr\"odinger equation with hinged boundary control and collocated observation. \textit{Math. Control Signals Systems} 28:22.


\bibitem {zz} Zheng, C., Zhongcheng, Z. (2012). Exact Controllability for the Fourth Order Schr\"oodinger Equation. \textit{Chin. Ann. Math.} 33:395--404.

\bibitem{zheng} Zheng, C. (2015). Inverse problems for the fourth order Schr\"odinger equation on a finite domain. \textit{Mathematical Control and Related Fields} 5:177--189. 

\bibitem {Zuazua} Zuazua, E. (1990). Exact controllability for the semilinear wave equation. \textit{J. Math. Pures Appl.} 69:33--55.

\end{thebibliography}
\end{document}